\documentclass[11pt,reqno]{amsart}
\usepackage[margin=1.2in,letterpaper]{geometry}                
\usepackage{graphicx}
\usepackage{amsmath,amssymb}
\usepackage{enumitem}
\usepackage{epstopdf}
\usepackage[colorlinks=true, pdfstartview=FitV, linkcolor=blue, citecolor=blue, urlcolor=blue]{hyperref}
\usepackage{xcolor}
\usepackage{tikz-cd}
\usepackage[nameinlink]{cleveref}
    \crefname{conj}{conjecture}{conjectures}
    \crefname{conj}{Conjecture}{Conjectures}

\numberwithin{equation}{section}

\newtheorem{thm}{Theorem}[section]

\newtheorem{thmx}{Theorem}

\newtheorem{conj}[thm]{Conjecture}
\newtheorem{cor}[thm]{Corollary}
\newtheorem{lem}[thm]{Lemma}
\newtheorem{prop}[thm]{Proposition}

\newtheorem{quest}[thm]{Question}

\theoremstyle{definition}
\newtheorem{defn}[thm]{Definition}
\newtheorem*{defn*}{Definition}

\newtheorem{ex}[thm]{Example}

\newtheorem{rem}[thm]{Remark}

\newcommand{\N}{{\mathbb N}}
\newcommand{\Q}{{\mathbb Q}}
\newcommand{\Z}{{\mathbb Z}}
\newcommand{\bP}{{\mathbf P}}

\newcommand{\cF}{{\mathcal F}} 
\newcommand{\R}{{\mathcal R}}
\newcommand{\K}{{\mathcal K}}
\newcommand{\I}{{\mathcal I}}
\newcommand{\J}{{\mathcal J}}
\newcommand{\cL}{{\mathcal L}}

\newcommand{\m}{{\mathfrak m}}

\newcommand{\sym}{\mathfrak{S}}

\renewcommand{\P}{\mathrm{P}}
\renewcommand{\L}{\mathrm{\Lambda}}

\def\codim{\operatorname{ht}}
\def\depth{\operatorname{depth}}

\def\pd{\operatorname{pd}}

\def\reg{\operatorname{reg}}

\def\Ass{\operatorname{Ass}}
\def\Ss{\operatorname{Ss}}
\def\Sss{\operatorname{Sss}}
\def\bor{\operatorname{Borel}}
\def\St{\operatorname{St}}
\newcommand{\pa}{\operatorname{part}}

\def\lex{{\operatorname{antilex}}}
\def\Min{{\operatorname{Min}}}

\def\Sym{{\operatorname{Sym}}}

\def\ov{\overline}
\renewcommand{\l}{\lambda}
\def\med{\operatorname{med}}
\def\supp{\operatorname{supp}}

\def\astab{\operatorname{astab}}
\def\dstab{\operatorname{dstab}}
\def\conv{\operatorname{conv}}
\def\np{\operatorname{np}}

\title[The combinatorial structure of symmetric strongly shifted ideals]{The combinatorial structure of \\ symmetric strongly shifted ideals}

\author{Alessandra Costantini and Alexandra Seceleanu}

\address[Alessandra Costantini]{Oklahoma State University}
\email{alecost@okstate.edu}

\address[Alexandra Seceleanu]{University of Nebraska-Lincoln}
\email{aseceleanu@unl.edu}


\date{}

\begin{document}

\maketitle 
\begin{abstract}
Symmetric strongly shifted ideals are a class of monomial ideals which come
equipped with an action of the symmetric group and are analogous to the well-studied class of strongly stable monomial ideals. In this paper we focus on algebraic and combinatorial properties of symmetric strongly shifted ideals. On the algebraic side, we elucidate properties that pertain to behavior under ideal operations, primary decomposition, and the structure of their Rees algebra. On the combinatorial side, we develop a notion of partition Borel generators which leads to connections to discrete polymatroids, convex 
polytopes, and permutohedral toric varieties.
\end{abstract}

\tableofcontents

\vspace{-2 em}

\section{Introduction}

Let $R=K[x_1,\ldots, x_n]$ be a polynomial ring over a field $K$ and consider  the set of partitions 
\[
\P_n=\{\l=(\lambda_1,\dots,\lambda_n) \in \mathbb{Z}^n: 0 \leq \lambda_1 \leq \lambda_2 \leq \cdots \leq \lambda_n\}.
\]
Orbits of monomials under the natural action of the symmetric group $\sym_n$ can be identified with elements of $P_n$. This induces a bijection between $\sym_n$-fixed monomial ideals $I\subset R$ and sets of partitions $P(I)\subset P_n$ given by  \[
 \P(I)=\{ \lambda \in \P_n:x^\lambda=x_1^{\l_1}\cdots x_n^{\l_n} \in I\}.
 \]

The central object of study of this note are the following classes of $\sym_n$-fixed monomial ideals, which were introduced in \cite{BDGMNORS}.

\begin{defn}
  \label{def:shifted}
  Let $I \subset R$ be an $\sym_n$-fixed monomial ideal. 
  We say that $I$ is a \textbf{symmetric shifted ideal}, or {\bf ssi}, if, 
  for every $\lambda =(\lambda_1,\dots,\lambda_n) \in \P(I)$ and $1 \leq i <n$ with $\lambda_i < \lambda_n$, one has $x^\lambda (x_i/x_n) \in I$. 
   We say that $I$ is a \textbf{symmetric strongly shifted ideal}, or {\bf sssi}, if, 
   for every $\lambda =(\lambda_1,\dots,\lambda_n) \in \P(I)$ and $1 \leq i < j \leq n$ with $\lambda_i < \lambda_j$, one has $x^\lambda (x_i/x_j) \in I$. 
  Monomials $x^\lambda (x_i/x_n)$ and $x^\l (x_i/x_j)$ satisfying the conditions above are referred to as being obtained from $x^\l$ by a \textbf{Borel move}.
\end{defn}

The definitions of symmetric shifted and strongly shifted ideals are inspired by the definition of stable and strongly stable ideals. These are the most important classes of monomial ideals in computational algebra since, e.g., in characteristic zero generic initial initials are strongly stable. 
Moreover, stable and strongly stable ideals have well-understood minimal graded free resolutions. These were constructed by Eliahou and Kervaire, who also gave a formula for their graded Betti numbers in terms of the data of their minimal systems of monomial generators \cite{EK}. Analogous results for symmetric shifted ideals were obtained in \cite{BDGMNORS}. 

In this article, we initiate a comprehensive study of the algebraic properties  and combinatorial structure of symmetric strongly shifted ideals. Along the way we point out numerous items which extend the  list of similarities between symmetric strongly shifted ideals and strongly stable ideals.
A key tool we develop is the notion of \emph{partition Borel generators} of a symmetric strongly shifted ideal (see \Cref{def:Borelgenerators}), which is inspired by the notion of Borel generators of a strongly stable ideal and allows us to interpret a sssi as the symmetrization of a strongly stable ideal  (we refer the reader to \Cref{section:Borelgens} for any unexplained terminology). More precisely, our first main result is the following. 

\begin{thmx}[cf.~\Cref{thm:sshifted vs sstable}]
\label{thmA}
An ideal $I$ is symmetric strongly shifted if and only if $I$ is the symmetrization of a strongly stable ideal $J$ in the following sense
\[
I=\bigcap_{\sigma\in\sym_n} \sigma(J).
\]
Moreover, under this symmetrization process, the partition Borel generators of $I$ correspond to the Borel generators of $J$.
\end{thmx}

While symmetric strongly shifted ideals are almost never strongly stable (see \Cref{rem:neversstable}), the symmetrization process described by \Cref{thmA} sometimes allows to transfer desirable algebraic properties from the class of strongly stable ideals to the class of symmetric strongly shifted ideals. 
As an exemplification of this principle, in \Cref{prop:normal} we identify classes of normal strongly shifted ideals which are the symmetrization of normal strongly stable ideals. Moreover, from a primary decomposition of a \emph{principal Borel ideal} (i.e., a strongly stable ideal with exactly one Borel generator) we can determine a primary decomposition of a \emph{principal Borel sssi's} (i.e., a symmetric strongly shifted ideal with exactly one partition Borel generator).

\begin{thmx}[cf.~\Cref{prop:intersection symbolic}]
\label{thmB}
   Let $I$ be a principal Borel symmetric strongly shifted ideal, with partition Borel generator $\l = (\l_1, \ldots, \l_n)$. A  primary decomposition of $I$ is given by
   \[
   I = \bigcap_{j=1}^n   \bigcap_{\sigma \in \sym_n} \sigma (x_1, \ldots, x_j)^{\sum_{i=1}^{j} \l_i} .
   \]
 \end{thmx}

The primary decomposition of \Cref{thmB} can be refined to provide an irredundant decomposition (\Cref{thm:irredundant}), which allows to determine the associated primes of ordinary powers of principal Borel sssi's (\Cref{thm:stableAss}). The decomposition can also be rearranged as an intersection of symbolic powers of square-free sssi's. The latter ideals coincide with the  \emph{square-free Veronese ideals} (see \Cref{rem:stconfig}) and are the symmetrizations of the monomial prime ideals of $R$ in the sense of \Cref{thmA}. 
In fact, they dictate the algebraic and combinatorial properties of symmetric strongly shifted ideals in a similar way as the monomial primes determine the algebraic and combinatorial properties of strongly stable ideals. 

On the algebraic side, these square-free sssi's determine the radical (\Cref{prop:radical}) and the symbolic powers (\Cref{prop:symbolic}) of every symmetric strongly shifted ideal. Moreover, every sssi with exactly one partition Borel generator can be factored as a product of square-free Veronese ideals (\Cref{prop:Borel=product}).
In turn, the latter factorization result gives rise to a combinatorial characterization of principal Borel sssi's in terms of \emph{discrete polymatroids}, introduced by Herzog and Hibi in \cite{HHibi02} as a generalization of the notion of matroid. A monomial ideal is called a \emph{polymatroidal ideal} if the exponent vectors of its monomial generators form a discrete polymatroid (see \Cref{def:polymatroidal ideal} for more details). 
Surprisingly, the class of principal Borel sssi's coincides with the class of symmetric polymatroidal ideals.

\begin{thmx}[cf.~\Cref{thm:polymatroidal}]
\label{thmC}
A monomial ideal is symmetric and polymatroidal if and only if it is a symmetric strongly shifted ideal with exactly one partition Borel generator.
\end{thmx}


The factorization property of a principal Borel sssi given by \Cref{prop:Borel=product} allows for a beautiful description of their \emph{toric ideal}, whose definition we recall in \Cref{section:Rees}. In more detail, our main result is the following.
\begin{thmx}[cf.~\Cref{prop:normalRees} and \Cref{thm:quadraticFR}]
\label{thmD}
The toric ideal of a symmetric strongly shifted ideal with exactly one partition Borel generator is generated by quadratic polynomials, namely the symmetric exchange relations.  Moreover, the toric ring is a Cohen-Macaulay normal domain which has rational singularities in characteristic zero and is $F$-rational in positive characteristic.
\end{thmx}
In particular, our result provides supporting evidence for a longstanding conjecture of White, Herzog and Hibi \cite{{White},{HHibi02}}, which states that, for an arbitrary polymatroidal ideal $I$, the toric ideal of $I$ is generated by the symmetric exchange relations. \Cref{thmD} shows that the conjecture holds for all \emph{symmetric} polymatroidal ideals. 

 From a geometric perspective, a principal Borel sssi defines a normal {\em toric permutohedral variety}, i.e., an algebraic variety associated with a well-studied convex polytope dubbed the {\em permutohedron} (see \Cref{prop:permutohedron} and \Cref{cor:normalpolytope}).

While the combinatorial structure of symmetric strongly shifted ideals with an arbitrary number of Borel generators remains more mysterious, the knowledge of their syzygies from \cite{BDGMNORS} offers a valuable complementary source of information. 
In particular, it allows us to determine the depths of powers of an equigenerated symmetric shifted ideal (\Cref{prop:depth shifted}) and to provide a structure theorem for the \emph{Rees algebra} $\R(I) \cong \oplus_{k \geq 1} I^k$ of an equigenerated symmetric strongly shifted ideal $I$. Namely, in \Cref{thm:fibertype} we prove that $\R(I)$ is the quotient of a polynomial ring modulo relations which are either linear or arise from the toric ideal of $I$. The latter statement is analogous to a well-known result on the Rees algebra of an equigenerated strongly stable ideal \cite[Theorem 5.1]{HHV}.


\smallskip

\paragraph{\bf Structure of the paper} In \Cref{section:operations} we analyze the behavior of symmetric (strongly) shifted ideals under various algebraic operations.
In \Cref{section:Borelgens} we introduce partition Borel generators and prove \Cref{thmA}. We study principal Borel sssi's in \Cref{section:principalBorel}, while in \Cref{section:polymatroidal} we prove \Cref{thmC} and identify some distinguished classes of symmetric polymatroidal ideals (\Cref{prop:SEP} and \Cref{prop:transversal}). In  \Cref{section:invariants} we give combinatorial formulas for several numerical invariants associated with a symmetric monomial ideal. 

In the rest of the article, we exploit the combinatorial structure of sssi's to study their ordinary and symbolic powers. In \Cref{section:normality} we discuss the normality property of symmetric strongly shifted ideals, which is used in \Cref{section:assprimes} to study the associated primes of their powers.
In \Cref{section:stableAss}  we prove \Cref{thmB} and identify irredundant primary decompositions of sufficiently large powers of principal Borel sssi's.
Finally, we study the Rees algebra of symmetric strongly shifted ideals in \Cref{section:Rees}, where we prove our structure theorems \Cref{thm:fibertype} and \Cref{thm:quadraticFR} and discuss the geometry of toric varieties associated with a principal Borel sssi. 

\smallskip

\paragraph{\bf Notational conventions} 
We say that a sequence $\lambda=(\lambda_1,\dots,\lambda_n)$ of non-negative integers is a {\em partition} of $d$ of length $n$, if $\lambda_1 \leq \cdots \leq \lambda_n$ and $|\mathbf \lambda|= \lambda_1+ \dots+ \lambda_n = d$. We opt for the less standard convention of ordering the parts nondecreasingly for our conventions to match those in \cite{BDGMNORS}, where symetric shifted ideals were originally introduced. If $\l$ has distinct parts $p_1, \ldots, p_s$ which occur with multiplicities $n_1, \ldots, n_s$ respectively, we sometimes use the alternate notation $\l=(p_1^{n_1}, \ldots, p_s^{n_s})$.  
 Throughout the paper the notation $e_i$ stands for the $i$-th standard basis vector in $\Z^n$.

For a monomial $u=x_1^{a_1}\cdots x_n^{a_n}$, we write $\pa(u) \in P_n$ for the partition obtained from
$(a_1,\dots,a_n)$ by ordering its entries non-increasingly. If a monomial ideal $I \subset R$ is $\sym_n$-fixed, then a monomial $u$ is in $I$ if and only if $x^{\pa(u)}$ is in $I$.  The set $P(I)$ contains a partition $\lambda \neq (0^n)$  if and only if $I \neq R$. Throughout the paper we assume that $I\neq R$.

\smallskip

\paragraph{\bf Acknowledgements}
This work was partially supported by an NSF-AWM Mentoring Travel Grant. The first named author thanks Federico Castillo, Chris Francisco, Jeff Mermin, Jonathan Monta\~no, Jay Schweig and Gabriel Sosa for insightful discussions about strongly stable and polymatroidal ideals. The second author is partially supported by NSF grant DMS--2101225.

\section{Symmetric shifted ideals under ideal operations}
\label{section:operations}


In \cite[Proposition 1]{Cimpoeas}, Cimpoea\c{s} observes that the class of ideals of Borel type (which generalizes strongly stable ideals) is closed under sum, intersection, product, and colon operations. 
The class of strongly stable ideals is also  closed under the same operations as demonstrated in \cite[Proposition 1.2]{GuoWu} and also under taking integral closure \cite[Theorem 2.1]{GuoWu} and symbolic powers \cite[Theorem 3.8]{GuoWu}. (In the latter work, symbolic powers are taken by retaining the primary components associated to {\em minimal primes} of $I$.)

In this section we show that a majority of these statements are also true for symmetric strongly shifted ideals and fewer also hold for symmetric shifted ideals. Towards this end, it will be convenient to simplify \Cref{def:shifted} slightly.
Denoting by $G(I)$ the set of  minimal  monomial generators of $I$,  allows to single out the {\em partition generators} of $I$, namely  
\[
\L(I)=\{\lambda \in P_n: x^\lambda \in G(I)\}.
\]
In view of \cite[Lemmas 2.2 and 2.3]{BDGMNORS} it suffices to check the conditions of \Cref{def:shifted} for the partition generators $\l\in \L(I)$ rather than for arbitrary elements $\l$ in $P(I)$.

We begin by proving that the class of symmetric shifted ideals is closed under sums and intersections.

\begin{prop}
 \label{prop:sum intersection}
    Let $I,J$ be symmetric (strongly) shifted ideals. Then $I+J$ and $I \cap J$ are symmetric (strongly) shifted ideals. 
\end{prop}
\begin{proof}
   We prove the statement assuming that $I$ and $J$ are symmetric strongly shifted. The proof for symmetric shifted ideals is analogous, hence left to the diligent reader.
   
   First, notice that $I+J$ and $I \cap J$ are $\sym_n$-fixed, by the symmetry of $I$ and $J$. Now, let $\l\in P_n$ be such that $\l_i<\l_j$. If $x^\l\in I+ J$, then $x^{\l} \in I$ or $x^{\l} \in J$. In the first case, $x^\l x_i/x_j\in I$ by the strongly shifted property of $I$, while in the second case $x^\l x_i/x_j\in J$ because $J$ is strongly shifted. In either case,  $x^\l x_i/x_j\in I+J$, which proves that $I+J$ is a sssi. 
   Similarly, if  $x^\l\in I\cap J$, then $x^\l x_i/x_j\in I$ and $x^\l x_i/x_j\in J$ by the strongly shifted property of $I$ and $J$. Hence, $x^\l x_i/x_j\in I\cap J$ and $I\cap J$ is symmetric strongly shifted.
\end{proof}

The following example shows that the class of symmetric shifted ideals is not closed under taking products. 
\begin{ex}
\label{ex:productnotshifted}
   In $k[x_1, x_2, x_3, x_4]$, consider the symmetric shifted ideal $I$ with 
   \[
   \Lambda (I) = \{ (1,1,2,2), (0,2,2,2), (0,1,2,3) \}
   \]
   Notice that $I$ is not a sssi, since $(0,1,2,3) \in P(I)$ but $(1,1,1,3) \notin P(I)$; see \cite[Example 2.5]{BDGMNORS}. Moreover, the maximal ideal $\m=(x_1, x_2, x_3, x_4)$ is symmetric strongly shifted, with $ \Lambda (\m) = \{ (0,0,0,1) \}$. Then, the monomial ideal $I\m$ is $\sym_n$-invariant with 
   \[
   \Lambda (I\m) = \{  (1,2,2,2), (1,1,2,3), (0,2,2,3),  (0,1,2,4), (0,1,3,3)  \}
   \]
  but is not symmetric shifted, since $(0,1,2,4) \in P(I\m)$ but $(1,1,1,4) \notin P(I\m)$.
\end{ex}

Nevertheless, the subclass of symmetric strongly shifted ideals is closed under products. 

\begin{prop}
\label{prop:product}
A product of symmetric strongly shifted ideals is symmetric strongly shifted.
\end{prop}
\begin{proof}
It suffices to prove the statement for the product of two symmetric strongly shifted ideals $I,J$. It is clear that the ideal
\[
IJ = \left(\{\sigma(x^\l)\tau(x^\mu) : \sigma, \tau\in \sym_n,  \l\in P(I), \mu\in P(J) \}\right)
\]
 is fixed by $\sym_n$. It remains to show $IJ$ is strongly shifted. Towards this end let $p\in P_n$ be such that $x^p=\sigma(x^\l)\tau(x^\mu)\in IJ$ be as above and denote the relevant monomials by $\sigma(x^\l)=x^{\sigma(\l)}\in I$ and $ \tau(x^\mu)=x^{\tau(\mu)}\in J$, respectively. Now assume that $p_i=(\sigma(\l)+\tau(\mu))_i<(\sigma(\l)+\tau(\mu))_j=p_j$. This implies that $\sigma(\l)_i<\sigma(\l)_j$ or $\tau(\mu)_i<\tau(\mu)_j$. 
 
Suppose the former case holds. The inequality $\sigma(\l)_i<\sigma(\l)_j$ can be written equivalently as $\l_{\sigma^{-1}(i)}<\l_{\sigma^{-1}(j)}$. Since $I$ is symmetric and $\sigma(x^\l)\in I$ we have that $x^\l\in I$. Since $I$ is additionally strongly shifted and $\l_{\sigma^{-1}(i)}<\l_{\sigma^{-1}(j)}$, one deduces that $x^\l x_{\sigma^{-1}(i)} /x_{\sigma^{-1}(j)} \in I$. 
Finally, applying $\sigma$ yields
$\sigma\left (x^\l x_{\sigma^{-1}(i)} /x_{\sigma^{-1}(j)} \right)=\sigma(x^{\l})x_i/x_j \in I$ and thus $x^p x_i/x_j =\sigma(x^\l)\tau(x^\mu)\in IJ$.  The latter case is identical, hence omitted.
\end{proof}

From \Cref{prop:product} we deduce that the class of symmetric strongly shifted ideals is closed under taking powers.
 
\begin{cor}
\label{cor:powers}
Powers of symmetric strongly shifted ideals are symmetric strongly shifted.
\end{cor}

\Cref{ex:productnotshifted} shows that \Cref{prop:product} cannot be extended to symmetric shifted ideals. However, we do not know of an example that would show that \Cref{cor:powers} cannot be extended to symmetric shifted ideals. Thus we are left with the following:

\begin{quest}
\label{quest:powers shifted}
Is the class of symmetric shifted ideals closed under taking powers? 
\end{quest}

\medskip
We next analyze the behavior of symmetric (strongly) shifted ideals under taking radicals and saturations. The following description of the square-free symmetric monomial ideals of $R$ will be crucial. 

\begin{rem}
\label{rem:stconfig}
Any square-free symmetric ideal is the ideal generated by all the square-free monomials of a fixed degree $d\in\N$. Such an ideal is referred to in the literature as the \emph{square-free Veronese ideal} of degree $d$. It can also be interpreted as the defining ideal of a {\em monomial star configuration}, introduced in \cite{GHM} and denoted $I_{n,c}$, since it can be described as 
\begin{equation}
\label{eq:stconfig}
I_{n,c}=\bigcap_{1\leq i_1<\cdots<i_c\leq n} (x_{i_1}, \ldots, x_{i_c}), 
\end{equation}
where $c=n-d+1$ is the height of the ideal.
 The ideals $I_{n,c}$ are symmetric strongly shifted, so in particular all square-free symmetric ideals are in fact symmetric strongly shifted. 
 \end{rem}

\begin{prop}
\label{prop:radical}
  The radical of a symmetric ideal $I$ is symmetric strongly shifted and is described as
\[
\sqrt{I}=\bigcap_{1\leq i_1<\cdots<i_c\leq n} (x_{i_1}, \ldots, x_{i_c}) = I_{n,c\,}, \quad
\text{ where } \quad
c=\codim(I).
\]
\end{prop}
\begin{proof}
  Notice that $\sqrt{I}$ is symmetric since $I$ is. As $\sqrt{I}$ is also square-free, the conclusion follows from \Cref{rem:stconfig}. 
\end{proof}

\Cref{ex:saturation shifted} below shows that the class of symmetric shifted ideals is not closed under colons or saturations. However, 
in \Cref{cor:saturation sssi} we prove that symmetric strongly shifted ideals are closed under saturation with respect to any symmetric monomial ideal. To prove this result, we first observe that taking colons preserves symmetry.

\begin{lem}
\label{lem:colonsymmetric}
If $I, J$ are symmetric ideals then $I:J$ is symmetric.
\end{lem}
\begin{proof}
Let $f \in I \colon J$. Since $I$ and $J$ are $\sym_n$-invariant, from $f J \subset I$ it follows that $\sigma(f) J = \sigma(f) \sigma(J) \subseteq \sigma(I)=I$ for all $\sigma \in \sym_n$, so $I \colon J$ is symmetric. 
\end{proof}

Our next observation is the fact that, unlike arbitrary symmetric shifted ideals, sssi's can be characterized combinatorially in terms of the so called dominance order. 

\begin{defn}
  Partitions $\lambda,\mu\in P_n$ are compared in the \textbf{dominance order} $\lhd$ by setting
 \[
 \mu \unlhd \lambda \,\text{ iff } \,\Sigma_k(\mu) \leq \Sigma_k(\l) \text{ for all } 1\leq  k\leq n.
 \]
 \end{defn}

\begin{rem}
\label{rem:dominance}
  An $\sym_n$-fixed monomial ideal $I$ is strongly shifted if and only if, for every $\lambda, \mu \in \P_n$ with $|\lambda|=|\mu|$, $\lambda \in P(I)$ and $\mu \unlhd \lambda $ imply $\mu \in \P(I)$. 
  This is because for $\l, \mu$ satisfying $|\lambda|=|\mu|$, the inequality  $\mu \unlhd \lambda $ is equivalent to $x^\mu$ being obtained from $x^\l$ by a sequence of Borel moves; see \cite[Lemma 1.3]{DeNegri} for a proof.
  \end{rem}

We can now describe the saturation of a symmetric strongly shifted ideal with respect to a square-free symmetric monomial ideal. Recall that the saturation operation on ideals is defined by $I:J^\infty=\bigcup_{i\geq 0} I: J^i$ and has the property that $I:J^\infty=I:J^N$ for $N\gg 0$.
\begin{prop}
\label{prop:colon}
Let $I$ be a sssi and let $c$ be a natural number so that $1\leq c\leq n$. 
For $\l\in P_n$ define the truncated partition  $\l_{<c}=(\l_1,\ldots, \l_{c-1})\in P_{c-1}$.
Then we have
\begin{enumerate}
\item $I:I_{n,c}^\infty=(\sigma(x^\mu) \mid \sigma\in \sym_n, \mu_{<c}=\l_{<c} \text{ for some } \l \in P(I))${\rm;}
\item $I:I_{n,c}^\infty$ is symmetric strongly shifted.
\end{enumerate}
\end{prop}
\begin{proof}
Denote 
$J=(\sigma(x^\mu) \mid \sigma\in \sym_n, \mu_{<c}=\l_{<c} \text{ for some } \l \in P(I))$.

(1) Since \Cref{lem:colonsymmetric} yields that $I:I_{n,c}^\infty$ is a symmetric monomial ideal, it suffices to show 
\[
P(I:I_{n,c}^\infty)=P(J)=\{\mu\in P_n\mid \mu_{<c}=\l_{<c} \text{ for some } \l \in P(I)\}.
\]
If $x^\mu\in I:I_{n,c}^\infty$, since $x_cx_{c+1}\cdots x_n\in I_{n,c}$, we have  $x^\mu(x_cx_{c+1}\cdots x_n)^N=x^\l$ for some $\l\in P(I)$ and $N\gg 0$. This implies that $\mu_{<c}=\l_{<c}$ and establishes the containment $I:I_{n,c}^\infty\subseteq J$.

Conversely, take $\mu\in P(J)$ and $\l \in P(I)$ with the property that $\mu_{<c}=\l_{<c}$. Take also $N\geq \l_n$. Then we have $x^\l \mid x^\mu(x_cx_{c+1}\cdots x_n)^N$ and therefore  $x^\mu(x_cx_{c+1}\cdots x_n)^N\in I$; thus we conclude $\mu+(0^{c-1}, N^{n-c+1})\in P(I)$. Now consider an arbitrary monomial $x^\alpha \in I_{n,c}^N$. We aim to show that $x^\mu x^\alpha \in I$ and equivalently that $\pa(x^{\mu+\alpha})\in P(I)$. To do so it suffices to show $\pa(x^{\mu+\alpha})\unlhd \mu+(0^{c-1}, N^{n-c+1})$ and use the fact that $I$ is symmetric strongly shifted. Set $p=\pa(x^{\mu+\alpha})$ and observe that $x^\alpha\in I_{n,c}^N$ implies $\alpha_i\leq N$  for each $1\leq i\leq n$, so that $p$ satisfies inequalities $p_i\leq \mu_i+N$ for each $1\leq i\leq n$. Consequently, we have 
\[
\sum_{i= k}^n p \leq \sum_{i= k}^n (\mu_i+N) =\sum_{i=k}^n \left(\mu+(0^{c-1}, N^{n-c+1}) \right)_i \text{ for } k\geq c.
\]
As for the case $k<c$, note that $|\alpha|=N(n-c+1)$, which yields 
\[
\sum_{i= k}^n p_i = \sum_{i= k}^n \pa(x^{\mu+\alpha})_i \leq \left(\sum_{i= k}^n \mu_i \right)+ |\alpha| =\sum_{i=k}^n \left(\mu+(0^{c-1}, N^{n-c+1}) \right)_i \text{ for } k< c.
\]
As discussed above, this yields $p\in P(I)$, concluding the proof of the containment $J \subseteq I:I_{n,c}^\infty$.

For (2) we will prove the equivalent assertion that $J$ is a sssi. Towards this goal, we identify the minimal generators of $J$: these are given by the partitions
\[
\Lambda(J)=\{ \mu\in P_n \mid  \mu_{<c}=\l_{<c} \text{ for some } \l \in \Lambda(I), \mu_i=\mu_{c-1} \text{ for } i\geq c\ \}.
\]
Let $\mu\in \Lambda(J)$ and consider $1\leq i<j \leq n$ so that $\mu_i<\mu_j$. By the description of $\Lambda(J)$ this yields the inequality $1\leq i<c-1$.  There exists $\l\in P(I)$ with $\mu_{<c}=\l_{<c}$. Set $\mu'=\mu+e_i-e_j$ and $\l'=\l+e_i-e_j$. Since $\mu_{<c}=\l_{<c}$ and $1\leq i<c-1$ it follows that $\l_i=\mu_i< \mu_{j}\leq \l_j$ as either $j<c$ and  $\mu_{j}= \l_j$ or $j\geq c$ and thus $\mu_j=\mu_{c-1}=\l_{c-1}\leq  \l_j$. Since $I$ is symmetric strongly shifted we have $\l'\in P(I)$ and $\mu'_{<c}=\l'_{<c}$, so we obtain $\mu'\in P(J)$, as desired.
\end{proof}


\begin{cor}
\label{cor:saturation sssi} 
   Let $I$ be a sssi and let $J$ be a symmetric monomial ideal with $\codim(J)=c$.
   Then the ideal  $I:J^\infty$ is symmetric strongly shifted and is described as $I: J^\infty = I: I_{n,c}^\infty$.
\end{cor}
\begin{proof}
   First, notice that $I: J^\infty = I: \sqrt{J}^{\,\infty}$, since every ideal contains a power of its radical. Moreover, by \Cref{prop:radical} it follows that $\sqrt{J}=I_{n,c}$. 
   Therefore, $I: J^\infty = I: I_{n,c}^\infty$ is symmetric strongly shifted by \Cref{prop:colon}(2). 
\end{proof}

Thanks to \Cref{prop:colon}, we can understand the symbolic powers of a sssi. In the literature two notions of symbolic powers make an appearance. They are defined in terms of the {\em associated  primes} of $I$, the set of which we denote $\Ass(I)$ and the   {\em minimal  primes} of $I$, the set of which we denote $\Min(I)$, respectively. The symbolic powers of an ideal $I$ are defined by some authors as
\begin{equation}
\label{eq:symbolicpowerMin}
I^{(m)_{\Min}}=\bigcap_{P\in \Min(I)} (I^mR_P\cap R),
\end{equation}
and by others  as
\begin{equation}
\label{eq:symbolicpowerAss}
I^{(m)_{\Ass}}=\bigcap_{P\in \Ass(I)} (I^mR_P\cap R).
\end{equation}
The two definitions agree for ideals without embedded primes, in which case we will denote symbolic powers simply as $I^{(m)}$. The second definition has the advantage that it satisfies $I^{(1)_{\Ass}}=I$ for all ideals $I$, while the first definition is more easily handled and more relevant in geometric contexts.
Both notions of symbolic powers can be computed by saturation, a point that proves relevant in \Cref{prop:symbolic} below.

 In \cite[Theorem 4.3]{BDGMNORS} it is proven that the symbolic powers of a square-free sssi are symmetric strongly shifted. 
 In the next proposition we show that more generally the symbolic powers of any sssi are symmetric strongly shifted.

\begin{prop}
 \label{prop:symbolic}
The symbolic powers of a symmetric strongly shifted ideal are symmetric strongly shifted.
\end{prop}
\begin{proof}
Let $I$ be a sssi. Let $\Min(I)$ denote the set of minimal associated prime ideals of $I$ and $\Ass^*(I)$ the union of the associated prime ideals of $I^n$ for all $n\geq 0$. It is known that $\Ass^*(I)$ is a finite set \cite{Bro2}. 

The $m$-th symbolic powers $I^{(m)_{\Min}}$ and $I^{(m)_{\Ass}}$ of $I$ can both be described as saturations:
\begin{subequations}
 \begin{align} 
 I^{(m)_{\Min}} =\bigcap_{P\in \Min(I)} (I^m R_P\cap R)= I^m:J^\infty  & \quad \text{ for } J =\bigcap_{P \in \Ass^*(I)\setminus \Min(I)} P  \label{eq:saturationMin}\\ 
I^{(m)_{\Ass}} = \bigcap_{P\in \Ass(I)} (I^m R_P\cap R) = I^m:J^\infty  & \quad  \text{ for } J =\bigcap_{P \in \Ass^*(I)\setminus \Ass(I)_\subseteq} P \label{eq:saturationAss}
  \end{align}
\end{subequations}
Here we denote
\[
 \Ass(I)_\subseteq=\{P: P\subseteq P' \text{ for some } P'\in  \Ass(I)\}.
\]

Observe that both ideals termed $J$ above are square-free. We now show they are symmetric. 
Indeed, consider the sets $S=\Ass(I), \Min(I), \Ass(I^m), \Ass^*(I)$. Each of them are closed under the action of the symmetric group, that is,  $P\in S$ if and only if $\sigma(P)\in S$ for all $\sigma\in \sym_n$. It follows that the set differences $\Ass^*(I)\setminus \Min(I)$ and $\Ass^*(I)\setminus \Ass(I)_\subseteq$ are also closed under the action of $\sym_n$ which implies that in both cases $J$ is symmetric. Since $J$ is symmetric and square-free, \Cref{rem:stconfig} yields that $J$ is a square-free Veronese ideal. 

To finish the proof, it then suffices to invoke \Cref{prop:colon}(2).
\end{proof}

We emphasize that \Cref{prop:colon} and \Cref{prop:symbolic} do not extend to symmetric shifted ideals which are not strongly shifted, as shown by the following example.

\begin{ex}
\label{ex:saturation shifted}
Let $I$ be the symmetric shifted ideal with $\Lambda(I)=\{(1,2,2), (0,2,3)\}$ and let $\m=(x_1,x_2,x_3)$. Note that $\m$ is symmetric strongly shifted, but $I$ is not, since $(0,2,3) \in P(I)$ but $(1,1,3) \notin P(I)$. Then the ideal
$
I^{(1)_{\Min}}=I:\m=I:\m^\infty=(x_1^2x_2^2,\,x_1^2x_3^2,\,x_2^2x_3^2)
$
is not symmetric shifted and nor is the second symbolic power
$I^{(2)_{\Min}}=(x_{1}^{2}x_{2}^{2}x_{3}^{2},\,x_{1}^{4}x_{2}^{4},\,x_{1}^{4}x_{3}^{4},\,x_{2}^{4}x_{3}^{4})$ in the sense of \eqref{eq:symbolicpowerMin}. 
\end{ex}

We do not currently know of any  ssi $I$ admitting a symbolic power $I^{(m)_{\Ass}}$ which is not symmetric shifted.

\medskip
Our last result in this section concerns the integral closure of symmetric shifted ideals. 

\begin{prop}
\label{prop:integralclosure}
The integral closure  $\ov{I}$ of a symmetric (strongly) shifted ideal $I$ is also symmetric (strongly) shifted.
\end{prop}
\begin{proof}
First note that the integral closure of a (not necessarily monomial) symmetric ideal is symmetric. Indeed if $f$ is a solution of a monic polynomial
\[
x^n+r_1x^{n-1}+\cdots +r_n=0
\]
with $r_i\in I^i$ and if $\sigma\in \sym_n$ then $\sigma(f)$ is a solution of the equation
\[
x^n+\sigma(r_1)x^{n-1}+\cdots +\sigma(r_n)=0
\] 
in which  each $\sigma(r_i)\in \sigma(I^i)=I^i$ because $I^i$ is stable under the action of $\sym_n$.

Second, recall that since $I$ is a monomial ideal, $x^\alpha\in \ov{I}$ if and only if it satisfies an equation of the form 
\begin{equation}
\label{eq:intcl}
x^{s\alpha}=\prod_{k=1}^s x^{ \beta^{(k)}} \text{ for some } s\in\N,
\end{equation}
where $x^{ \beta^{(k)}}\in I$ for each $k$; see \cite[p.~9]{HSwanson}. Now assume that $\alpha\in P(\ov{I})$ and take $i<j$ so that $\alpha_i<\alpha_j$ (if $I$ is shifted, assume additionally that $j=n$). Set $S=\left \{k\in \{1,\ldots, s\} \mid \beta^{(k)}_i< \beta^{(k)}_j\right \}$ and $\alpha'=\alpha+e_i-e_j$. From \eqref{eq:intcl} it follows that 
\[
\sum_{k\in S} (\beta^{(k)}_j- \beta^{(k)}_i)\geq s(\alpha_j-\alpha_i)\geq s.
\]
Therefore there exist nonnegative integers $c_k$, one for each $k\in S$, so that $\sum_{k\in S} c_k=s$ and $c_k\leq \beta^{(k)}_j- \beta^{(k)}_i$ for each $k\in S$. 
Define $\alpha' := \alpha+e_i-e_j$ and 
\[
\gamma^{(k)}=\begin{cases}
\beta^{(k)}+c_k(e_i -e_j) & \text{ if }k\in S\\
\beta^{(k)} & \text{ otherwise}
\end{cases}
\]
so that $\sum_{k=1}^n \gamma^{(k)}=\sum_{k=1}^n \beta^{(k)}+s(e_i-e_j)=s\alpha'$.
Then $\gamma^{(k)}\in I$ by the (strongly) shifted property and the relation
\[
x^{s\alpha'}=\prod_{k=1}^s x^{ \gamma^{(k)}}
\]
shows that $x^{\alpha'}\in \ov{I}$, concluding the proof.
\end{proof}

\section{Partition Borel generators and combinatorial structure}
\label{section:Borelgens}


The proofs of \Cref{prop:colon} and \Cref{prop:symbolic}, together with \Cref{ex:saturation shifted}, suggest that the algebraic differences between sssi's and symmetric shifted ideals which are not strongly shifted might be due to the combinatorial characterization of sssi's in terms of the dominance order (\Cref{rem:dominance}). 

In this section, we further investigate the role of the dominance order in determining the algebraic structure of symmetric strongly shifted ideals, by introducing the key notion of \emph{partition Borel generators} of a sssi (see \Cref{def:Borelgenerators}). 
As this definition is inspired by the notion of Borel generators of a strongly stable ideal, we begin by recalling relevant background information on strongly stable ideals which we will need throughout this section. 
\begin{defn}
  \label{def:stronglystable}
 A monomial ideal $I \subset R$ is said to be a \textbf{strongly stable ideal} if, 
  for every $\alpha  \in \N^n$ with $x^\alpha \in I$ and for every pair $i<j$ so that $\alpha_j\neq 0$ one has $x^\alpha x_i/x_j\in I$. 
  
  A monomial ideal $I \subset R$ is said to be a \textbf{stable ideal} if, 
  for every $\alpha  \in \N^n$ with  $x^\alpha \in I$, setting $\max(\alpha)=\max\{j \colon \alpha_j\neq 0\}$, for every $i<\max(\alpha)$  one has $x^\alpha x_i/x_{\max(\alpha)} \in I$. 
  
 The monomials $x^\alpha (x_i/x_j)$ and $x^\alpha (x_i/x_{\max(\alpha)})$ occurring above are referred to as being obtained from $x^{\alpha}$ via a \textbf{Borel move}. 
\end{defn}

Despite the formal analogy between \Cref{def:shifted} and \Cref{def:stronglystable}, symmetric strongly shifted ideals are almost never strongly stable. 
\begin{rem}
\label{rem:neversstable}
The only $\sym_n$-fixed ideals which are also strongly stable are the powers of the homogeneous maximal ideal. Indeed, if $I$ is strongly stable in $K[x_1, \ldots, x_n]$, then it must contain $x_1^d$, where $d\in \N$ is the least degree of a nonzero element of $I$. Hence by symmetry $I$ must contain $x_n^d$. Now the strongly stable property yields $(x_1, \ldots, x_n)^d\subseteq I$, which must be an equality since $d$ was the initial degree of $I$.
\end{rem}

Given a set of monomials $M$, the smallest strongly stable monomial ideal which contains $M$ is called the \textbf{Borel ideal} generated by $M$ and is denoted by $\bor(M)$. The following well-known result conveniently describes the set of monomials obtained from $x^\alpha$ by performing Borel moves, that is, the monomials in the Borel ideal $\bor(\{x^\alpha\})$.
 
\begin{rem}[{\cite[Lemma 1.3]{DeNegri}}]
\label{rem:Borelorder}
For each  $\alpha=(\alpha_1,\dots,\alpha_n)\in \N^n$, define $\Sigma_k(\alpha)=  \alpha_k + \cdots + \alpha_n$.
Monomials $x^\alpha$ and $x^\beta$ with $|\alpha|=|\beta|$ satisfy $x^\alpha\in \bor(\{x^\beta\})$ if and only if 
$\Sigma_k(\alpha) \leq \Sigma_k(\beta)$  for all  $1\leq  k\leq n$. We denote this condition by $x^\alpha \prec_B x^\beta$. This defines a partial order on the set of $n$-tuples of nonnegative integers which restricts to the dominance order on $P_n$.
\end{rem}

The relationship between strongly stable ideals and symmetric strongly shifted ideals is best understood in terms of the following notion, which can be interpreted as a symmetric analogue to the notion of Borel generators of a strongly stable ideal.

\begin{defn}
  \label{def:Borelgenerators}
  Let $B \subseteq P_n$ be a set of partitions and set  $x^B = \{x^\l: \l\in B\}$. 
  
  The \textbf{symmetric shifted ideal generated by $B$}, denoted $\Ss(B)$, is defined to be the smallest symmetric shifted ideal which contains $x^B$. Similarly, define the \textbf{symmetric strongly shifted ideal generated by $B$}, denoted $\Sss(B)$, to be the smallest symmetric strongly shifted ideal which contains $x^B$. 
  
 Conversely, for a symmetric strongly shifted (resp. shifted) ideal $I$ we define the set of  \textbf{partition Borel generators} of $I$, denoted $B(I)$, to be the smallest $B\subseteq P_n$ so that $I=\Sss(B)$ (resp. $I= \Ss(B)$).
\end{defn}

In analogy with \Cref{rem:Borelorder} for strongly stable ideals, the partition Borel generators of a symmetric strongly shifted ideal coincide with the maximal elements in each degree of $\L(I)$ with respect to the dominance order. More precisely, if for a set $C \subseteq P_n$ $\max_\unlhd \{C\}$ denotes the set of maximal elements of $C$ in dominance order, we have the following characterization. 
\begin{prop}
\label{prop:Borelgens}
  \begin{enumerate}
   \item Let $B \subseteq P_n$ be a set of partitions, then 
   \[
  \Sss(B)= \left(\{\sigma(x^\mu) : \sigma\in \sym_n, \ \mu \in P_n,  \ \exists \l\in B \text{ such that }  |\lambda|=|\mu| \text{ and } \ \mu \unlhd \lambda \}\right).
  \]
  \item Let $I$ be a symmetric strongly shifted ideal. Then, $B(I)=\max_\unlhd\{\Lambda(I)\}$, that is, 
\[
B(I)=\{ \lambda \in \L(I) : \mu \in \Lambda(I) \text{ with } |\lambda|=|\mu|, \  \lambda \unlhd \mu \text{ implies } \l=\mu\}.
\]
are the partition Borel generators of $I$. 
\end{enumerate}
In particular, for any symmetric strongly shifted ideal $I$ one has $I=\Sss(B(I))$.
\end{prop}
\begin{proof}
 (1)  Denote  $ I = \left( \{\sigma(x^\mu) : \sigma\in \sym_n, \ \mu \in P_n,  \ \exists \l\in B \text{ such that }  |\lambda|=|\mu| \text{ and } \ \mu \unlhd \lambda \} \right)$. 
 It is clear that $I$ is symmetric and $ x^B\subseteq I$. To show that $I$ is strongly shifted, notice that $\L(I) \subseteq S$, where
 \[
 S=\{\mu :\mu \in P_n,  \ \exists \l\in B \text{ such that }  |\lambda|=|\mu| \text{ and } \ \mu \unlhd \lambda \} \subset I.
 \]
 and $S$ is closed under Borel moves. By \Cref{rem:dominance} this ensures that $I$ is a sssi.
Finally, $I$ is indeed the smallest symmetric strongly shifted ideal containing $x^B$, since any sssi $I'$ with $I' \supseteq x^B$ must contain $x^S$ by \Cref{rem:dominance}, hence it must contain $I$ by symmetry.

(2) Set $B=\max_\unlhd\{\Lambda(I)\}$. Since $B\subseteq \Lambda(I)$, it is clear that $I \supseteq x^{B}$, whence $I \supseteq \Sss(B)$ since $I$ is an sssi. Moreover, by the definition of $B$, for each $\mu\in \Lambda(I)$ we have that $\mu\unlhd\l$ for some $\l\in B(I)$, hence $I \subseteq \Sss(B)$ by \Cref{rem:dominance}, so equality holds. That $B$ is the smallest set of Borel generators of $I$ follows  by noting that  $\l \in B(I)$ and  $B'\subset B\setminus\{\l\}$ yields  $\Sss(B')\subset I\setminus\{x^\l\}$. Thus we conclude that $B=B(I)$.
\end{proof}

\medskip

For a set of partitions $B \subseteq P_n$, let $\bor(x^B)$ be the strongly stable ideal generated by $x^B$. The next result clarifies how the ideal $\Sss(B)$ relates with $\bor(x^B)$. It provides a containment between these ideals,  which becomes an equality when one restricts to monomials whose exponents are partitions.


\begin{prop}
\label{prop:partitions sstable}
    Let $B \subseteq P_n$ be a set of partitions. Then, $\Sss(B)\subseteq \bor(x^B)$ and 
    \begin{equation}\label{eq:claim}
    P\left(\Sss(B)\right)= \{\l\in P_n: x^\l\in \bor(x^B)\}.
 \end{equation}
\end{prop}
\begin{proof}
Let  $\sigma(x^\l)\in \Sss(B)$ for some $\sigma\in \sym_n, \l\in P_n$. 
Since $\l$ is a partition, we have $\displaystyle{\Sigma_k (\sigma(\l)) = \sum_{i=k}^{n} \l_{\sigma^{-1}(i)} \leq \sum_{i=k}^{n} \l_i =  \Sigma_k(\l)}$, which means  that $\sigma(x^\l) \prec_B x^\l$ by  \Cref{rem:Borelorder}. Furthermore, since $x^\l\in \Sss(B)$, \Cref{rem:dominance} together with \Cref{rem:Borelorder} then imply that $\sigma(x^\l)\in\bor(\{x^\l\})\subseteq \bor(x^B)$ for each $\sigma \in \sym_n$, which justifies the ideal containment. 

%

The first part of the proof yields the containment $P\left(\Sss(B)\right)\subseteq  \{\l\in P_n: x^\l\in \bor(x^B)\}$. 

For the converse, let $\l\in P_n$ with $x^\l\in \bor(x^B)$. Then there exists  a minimal generator of $\bor(x^B)$, $x^\alpha$, so that $x^\alpha|x^\l$ and 
$x^\alpha\prec_B x^\beta$ for some $\beta\in B$ with $|\alpha|=|\beta|$. This yields $\l_i\geq \alpha_i$ for all $1\leq i \leq n$ and $\sum_{i=k}^n\alpha_i\leq \sum_{i=k}^n \beta_i$ for all $1\leq k\leq n$.
Note that this implies 
\begin{equation}
\label{eq:lab}
\sum_{i=1}^k\l_i\geq \sum_{i=1}^k \alpha_i\geq \sum_{i=1}^k\beta_i \text{ for each } 1\leq k\leq n.
\end{equation}

 If $|\l|=|\alpha|$ then $\l=\alpha$, which yields $\l\unlhd \beta$ and thus $x^\l\in \Sss(B)$.

Assume now that $|\l|>|\alpha|=|\beta|$ and set $q=\max\{k: \sum_{i=1}^k\l_i\leq |\beta|\}$ and $t=|\beta|-\sum_{i=1}^q\l_i$. By the assumption we have $q<n$. Moreover, since $\sum_{i=1}^{q+1}\l_i>|\beta|$ we deduce that $\l_{q+1}>t$. Now consider the vector
\[
\gamma=(\beta_1, \ldots, \beta_q, \beta_{q+1}+\l_{q+1}-t, \beta_{q+2}+\l_{q+2}, \ldots, \beta_{n}+\l_{n}).
\]
Note that $\gamma$ is a partition since $\beta, \lambda \in P_n$ and since $\l_{q+1}-t>0$. Moreover, observe that $x^\beta\mid x^\gamma$, which yields $x^\gamma\in \Sss(B)$ since $x^\beta\in \Sss(B)$. The sum of entries of $\gamma$ is
\[
|\gamma|=\sum_{i=1}^n\beta_i+\sum_{i=q+1}^n \l_i-t=|\beta|+\sum_{i=q+1}^n \l_i-|\beta|+\sum_{i=1}^q \l_i=|\l|.
\]
We claim that $\l \unlhd \gamma$. To see this, compute
\[
\sum_{i=k}^n\gamma_i =
\begin{cases}
\sum_{i=k}^n\beta_i + \sum_{i=k}^n\lambda_i  & \text{ if } k>q+1\\
\sum_{i=k}^n\beta_i + \sum_{i=q+1}^n\lambda_i -t   & \text{ if } k\leq q+1.
\end{cases}
\]
If $k>q+1$ it is clear that $\sum_{i=k}^n\gamma_i \geq \sum_{i=k}^n\l_i$. If $k\leq q+1$, then inequality \eqref{eq:lab} yields
\[
\sum_{i=k}^n\gamma_i =\sum_{i=k}^n\beta_i +\sum_{i=q+1}^n\lambda_i -|\beta|+ \sum_{i=1}^q\lambda_i =|\l|- \sum_{i=1}^{k-1}\beta_i \geq |\l|- \sum_{i=1}^{k-1}\l_i=\sum_{i=k}^{n}\l_i.
\]
Since we have shown $\l\unlhd \gamma$ and $x^\gamma\in \Sss(B)$, we deduce that $x^\l\in\Sss(B)$ by \Cref{rem:dominance} and thus $\l\in P(\Sss(B))$, as desired.

\end{proof}

The conclusion of \Cref{prop:partitions sstable} does not hold if one replaces $\Sss(B)$ with $\Ss(B)$ and $\bor(x^B)$ with the smallest stable ideal containing $x^B$, denoted $\St(x^B)$.

\begin{ex}
Let $B = \{ (0,0,1,1) \} $, $J= \St(x^B)= (x_3x_4, x_3^2, x_2x_3, x_1x_3)$ and let $I = \Ss ( \{\l\in P_n: x^\l\in J\})$. 
Observe that $I \nsubseteq J$, since $x_1x_2 \in I \setminus J$.
\end{ex}

 Building upon the relationship between $\Sss(B)$ and $\bor(x^B)$, \Cref{thm:sshifted vs sstable}  shows that an ideal $I$ is an sssi if and only if it can be obtained from a strongly stable ideal $J$ by {\em symmetrization}, i.e.,  
$I=\bigcap_{\sigma \in \sym_n} \sigma(J)$.

\begin{thm}
 \label{thm:sshifted vs sstable}
 \begin{enumerate}
     \item Let $J$ be a strongly stable ideal. Then the ideal
        \begin{equation}
        \label{eq:symmetrization}
        I = \bigcap_{\sigma \in \sym_n} \sigma(J)
\end{equation}
        is symmetric strongly shifted, with $P(I)= \{\l \in P_n : x^{\l} \in J \}$.
     \item  Conversely, every symmetric strongly shifted ideal $I$ has the form in \eqref{eq:symmetrization} for some strongly stable ideal $J$. In detail, let $B=B(I) \subset P_n$. Then, the ideal $J= \bor(x^B)$ satisfies        \[
        I = \bigcap_{\sigma \in \sym_n} \sigma(J).
        \]
    Moreover, $J$ is the smallest strongly stable ideal with this property. 
 \end{enumerate}
\end{thm}
\begin{proof}
(1) It is clear that $I$ is symmetric and that $P(I) \subseteq \{\l \in P_n : x^{\l} \in J \}$. Moreover, since $J$ is strongly stable, it follows from \Cref{prop:partitions sstable}  that  for each $\sigma \in \sym_n$, $\sigma(x^\l) \in \bor(\{x^\l\}) \subseteq J$ and consequently $x^\l\in I$ whenever $\l\in P_n$ and $x^\l \in J$. Therefore we have $\{\l \in P_n : x^{\l} \in J \} \subseteq P(I) $ and thus equality holds. Since $J$ is closed under Borel moves, it then follows that $P(I)$ is also closed under Borel moves, i.e. $I$ is a sssi.

(2) From \Cref{prop:partitions sstable} we know that $\Sss(B)\subseteq \bor(x^B)$. Applying permutations to this containment, we obtain that $\Sss(B)=  \sigma (\Sss(B)) \subseteq \sigma (\bor(x^B))$ for every $\sigma\in \sym_n$. Therefore, $\Sss(B) \subseteq \bigcap_{\sigma\in \sym_n} \sigma (\bor(x^B))$.  For the opposite containment, notice that $x^\alpha\in \bigcap_{\sigma\in \sym_n} \sigma (\bor(x^B))$ implies that $\sigma(x^\alpha)\in \bor(x^B)$ for all $\sigma \in \sym_n$ and in particular that $x^{\pa(x^\alpha)}\in \bor(x^B)$. Setting $\l=\pa(x^\alpha)$ we deduce by \Cref{prop:partitions sstable}  that $\l\in P\left(\Sss(B)\right)$. It follows by the $\sym_n$-invariance of $\Sss(B)$ that also  $x^\alpha\in \Sss(B)$, thus finishing the proof. 
\end{proof}

In light of \Cref{thm:sshifted vs sstable}, it is natural to expect that algebraic properties which are well-behaved under taking intersections are preserved under symmetrization.  \Cref{prop:normal} and \Cref{prop:intersection symbolic} will present instances when this is indeed the case.

\medskip
Analogously to \cite[Propositions 2.15, 2.16 and 2.17]{FMS11} for strongly stable ideals, 
our next goal is to describe sums, intersections and products of symmetric strongly shifted ideals in terms of partition Borel generators. To do so we must consider the lattice structure of $P_n$.

The following lemma is inspired by a similar description in \cite{Brylawski} for the lattice structure on the set of partitions of {\em fixed sum} $P(d)=\{ \l$ is a partition and $|\l|=d\}$. We include the proof here for lack of a specific reference which treats the case of $P_n$, the set of partitions with $n$ parts.

\begin{lem}
\label{lem:lattice}
    For every $n$, the set $P_n $ forms a lattice with respect to the dominance order.
\end{lem}
\begin{proof}
Notice that there is a one-to-one correspondence between partitions $\l = (\l_1, \ldots, \l_n) \in P_n$ and non-decreasing sequences $\hat{\l} \in \mathbb{N}^n$ so that $\hat{\l}_k + \hat{\l}_{k-2} \leq 2\hat{\l}_{k-1}$ for every $k$. Indeed, given $\l = (\l_1, \ldots, \l_n) \in P_n$ one defines $\hat{\l} = (\sum_{n}\l, \sum_{n-1}\l, \ldots, \sum_{1}\l)$, where for every $1 \leq k \leq n$ we denote $\sum_{k} \l = \l_k+ \ldots+ \l_n$. Conversely, given a vector $\hat{\l} \in \mathbb{N}^n$, one defines a partition $\l = (\l_1, \ldots, \l_n) \in P_n$ by setting $\l_n = \hat{\l}_1$ and $\l_{k} = \hat{\l}_{n-k+1} - \hat{\l}_{n-k}$ for $1 \leq k \leq n-1$. Under this correspondence, one has that $\eta \unlhd \l$ if and only if $\hat{\eta}_k \leq \hat{\l}_k$ for each $k$. 

Now, let $\l, \mu \in P_n$. We need to prove that $\l$ and $\mu$ have a meet and a join in the dominance order. For each $1 \leq k \leq n$, let $v_k = \min \{\hat{\l}_k, \hat{\mu}_k\}$ and let $v = (v_1, \ldots, v_n)$. Observe that for every $k$, $v_k + v_{k-2} \leq 2 v_{k-1}$, so there exists a partition $\eta \in P_n$ corresponding to $v$ under the identification above. Moreover, $v = \hat{\l} \wedge \hat{\mu}$ with respect to the componentwise order, hence $\eta = \l \wedge \mu$ with respect to the dominance order. 
To prove that $\l$ and $\mu$ have a meet, recall that for partitions $\l, \eta \in P_n$ one has that $\l \unlhd \eta$ if and only if $\eta^T \unlhd \l^T$, where $\eta^T$ and $\l^T$ denote the partitions corresponding to the transpose of the Young diagrams of $\eta$ and $\l$ respectively. Therefore, the join of $\l$ and $\mu$ is $\l \vee \mu = ( \l^T \wedge \mu^T)^T$. 
\end{proof}

In the next two propositions, $\max_\unlhd \{C\}$ denotes the set of maximal elements of a set $C\subset P_n$ in dominance order.
\begin{prop}
\label{prop:Borelmeet}
 Let $A,B\subseteq P_n$. Then
 \begin{enumerate}
    \item $\Sss(A) + \Sss(B) = \Sss(A \cup B)$ and
    \item $\Sss(A)\cap\Sss(B) = \Sss(A\wedge B)$, where $A \wedge B=\{\l \wedge \mu : \l\in A, \mu \in B\}$ and $ \l \wedge \mu$ denotes the meet  of $\l$ and $\mu$ in the dominance order. 
 \end{enumerate}
For any two symmetric strongly shifted ideals $I$ and $J$ one has 
\[
B(I+J) = \max_\unlhd\{B(I) \cup B(J)\} \text{ and } B(I \cap J) = \max_\unlhd \{B(I) \wedge B(J)\}.
\]
\end{prop}
\begin{proof}
For (1), it is clear that $x^{\l}\in \Sss(A) + \Sss(B)$ whenever $\l\in A \cup B$. Hence, since symmetric strongly shifted ideals are closed under sums by \Cref{prop:sum intersection}, we deduce that $\Sss(A \cup B) \subset \Sss(A) + \Sss(B)$. 
Conversely, let $p \in P_n$ be so that $x^p \in \Sss(A)+ \Sss(B)$. Since $x^p$ is a monomial, we must then have $x^p \in \Sss(A)$ or $x^p \in \Sss(B)$. In either case there exists a partition $\l\in A \cup B$ with $p \unlhd \l$. This implies that $x^p \in \Sss(A \cup B)$.

For (2), notice that $x^{\l\wedge\mu}\in \Sss(A) \cap \Sss(B)$ whenever $\l\in A$ and $\mu\in B$, since $\l\wedge\mu \unlhd \l$ and $\l\wedge\mu \unlhd \mu$ and since  $\Sss(A) \cap \Sss(B)$ is symmetric strongly shifted by  \Cref{prop:sum intersection}. We deduce that $\Sss(A\wedge B)\subseteq \Sss(A) \cap \Sss(B)$. 
Conversely, let $p \in P_n$ be so that $x^p \in \Sss(A) \cap \Sss(B)$. Then, there exist partitions $\l\in A, \mu \in B$ so that $p \unlhd \l$ and $p \unlhd \mu$. Since $\l$ and $\mu$ have a meet by \Cref{lem:lattice}, it then follows that $p \unlhd \l \wedge \mu$, so $x^p \in \Sss(A \wedge B)$. 

The remaining statement follows from the other two, since for any sssi's $I$ and $J$
\[
\Sss(B(I) \cup B(J)) = \Sss(B(I)) + \Sss(B(J)) = I+J
\]
and 
\[
\Sss(B(I) \wedge B(J)) = \Sss(B(I)) \cap \Sss(B(J)) = I \cap J.
\]
This implies that $\max_\unlhd \{B(I)  \cup B(J)\})=\max_\unlhd \{\L(I+J)\}=B(I+J)$ and $\max_\unlhd \{B(I) \wedge B(J)\})=\max_\unlhd \{\L(I \cap J)\}=B(I \cap J)$, which completes the proof.
\end{proof}

\begin{prop}
\label{prop:Borelsum}
For $A,B, C\subseteq P_n$ set $A+B=\{\l+\mu : \l\in A, \mu \in B\}$. Then we have
\begin{enumerate}
\item $\Sss(A) \cdot \Sss(B)= \Sss(A+B)$
\item $B(IJ)=\max_\unlhd \{B(I)+B(J)\}$ for any symmetric strongly shifted ideals $I, J$.
\end{enumerate}
\end{prop}
\begin{proof}
For (1), it is clear that $x^{\l+\mu}\in \Sss(A) \cdot \Sss(B)$ whenever $\l\in A$ and $\mu\in B$. Therefore, using the fact that  $\Sss(A) \cdot \Sss(B)$ is symmetric strongly shifted by  \Cref{prop:product}, we deduce that $\Sss(A+B)\subseteq \Sss(A) \cdot \Sss(B)$. Conversely, let $p \in P_n$ be so that $x^p$ is a monomial generator of $\Sss(A) \cdot \Sss(B)$. Then, by \Cref{prop:Borelgens}~(1) 
there exist partitions $\l\in A$, $\mu\in B$ and $\l',\mu'\in P_n$ and permutations $\sigma,\tau\in \sym_n$ so that $\l'\unlhd \l$, $\mu'\unlhd \mu$, and $p=\sigma(\l')+\tau(\mu')$.
Since $\l',\mu'$ are ordered increasingly, for all $k$ one has that 
\begin{eqnarray*}
p_k+\cdots +p_n &=& \l'_{\sigma^{-1}(k)}+ \ldots +\l'_{\sigma^{-1}(n)}+\mu'_{\tau^{-1}(k)}+ \ldots +\mu'_{\tau^{-1}(n)} \\
&\leq & \l'_{k}+ \ldots +\l'_{n}+\mu'_{k}+ \ldots +\mu'_{n} \\
& \leq & \l_{k}+ \ldots +\l_{n}+\mu_{k}+ \ldots +\mu_{n},
\end{eqnarray*}
where the last inequality follows from the fact that $\l'\unlhd \l$ and $\mu'\unlhd \mu$. Thus $p \unlhd \l + \mu$, whence 
$x^p\in \Sss(A+B)$, which completes the proof.

For (2), notice that part (1) and the fact that $I,J$ are strongly shifted imply the identity
\begin{eqnarray*}
\Sss(B(I)+B(J))=\Sss(B(I))\cdot \Sss(B(J)) =IJ .
\end{eqnarray*}
This yields the desired conclusion by ensuring that $\max_\unlhd \{B(I)+B(J)\})=\max_\unlhd \{\L(IJ)\}=B(IJ)$.
\end{proof}

\subsection{Principal Borel sssi's and permutohedra} 
\label{section:principalBorel}

Strongly stable ideals with only one Borel generator are called \emph{principal Borel ideals}. They play a special role in the theory of strongly stable ideals, due to their rich combinatorial structure and to the fact that every strongly stable ideal is a sum of principal Borel ideals.  Analogously, we introduce the following notion. 

\begin{defn}
  \label{def:principalBorelsss}
  A {\bf principal Borel symmetric shifted ideal} is any ideal of the form $ \Sss(\{\lambda\})$ for some $\l\in P_n$, i.e., an ideal whose set of partition Borel generators is a singleton.
 \end{defn}
  
Note that principal Borel sssi's are necessarily generated in a single degree (equigenerated). 

\begin{ex}
Examples  of principal Borel sssi's include the following:  
\begin{itemize}[leftmargin=5.5mm]
   \item  Powers of the maximal ideal $(x_1, \ldots, x_n)^d=\Sss(\{(0^{n-1}, d)\})$.
   \item  The square-free Veronese ideal of degree $n-c+1$ in $n$ variables; see \Cref{rem:stconfig}
   \[I_{n,c}=\Sss(\{(0^{c-1}, 1^{n-c+1})\}).\]
   \item  Powers $I_{n,c}^m=\Sss(\{(0^{c-1}, m^{n-c+1})\})$ of a square-free Veronese ideal.  
   The formula for the unique partition Borel generator follows from \Cref{prop:Borelsum}.
   \end{itemize}
\end{ex}

It is clear from \Cref{prop:Borelmeet}(1) that every sssi is a sum of principal Borel sssi's and that the class of principal Borel sssi's is not closed under sums.  However, it is closed under intersections and products.

\begin{cor}
\label{cor:Borelclosed}
Let $I, J$ be principal Borel sssi's. Then $I \cap J$ and $IJ$ are principal Borel sssi's. 
In particular, principal Borel sssi's are closed under taking powers. 
\end{cor}
\begin{proof}
Let $I=\Sss(\{\l\})$ and $J=\Sss(\{\mu\}$ for some $\l,\mu\in P_n$. Then, \Cref{prop:Borelmeet}(2) and \Cref{prop:Borelsum}(1) imply that $\Sss(\{\l\}) \cap \Sss(\{\mu\})=\Sss(\{\l \wedge \mu\})$ and $\Sss(\{\l\}) \Sss(\{\mu\})=\Sss(\{\l+\mu\})$ are principal Borel sssi's. 
\end{proof}

 \medskip
 A remarkable consequence of \Cref{prop:Borelsum} is that a principal Borel sssi decomposes as a product of square-free Veronese ideals.
Recall that the transpose of $\l$, denoted by $\l^T$, is defined as the partition corresponding to the transpose of the Young diagram of $\l$, so that the parts of $\l^T$ record the number of boxes in each row of the Young diagram of $\l$. More precisely $\l^T\in P_{\lambda_n}$ is defined by  
\[
\l^T_i=|\{j: \l_j\geq \lambda_n-i+1\}|\leq n.
\]

\begin{prop}
\label{prop:Borel=product}
Let $\l\in P_n$ be a partition and set $\l_0=0$. The principal Borel sssi $\Sss(\{\l\})$ decomposes as
\[
\Sss(\{\l\})=\prod_{i=1}^n I_{n,i}^{\l_{i}-\l_{i-1}}.
\]
\end{prop}
\begin{proof}
By the definition of $\l^T$, $\l$ can be decomposed as the sum of vectors corresponding to the rows of its Young diagram, that is,
$
\l=\sum_{j=1}^{\l_n} (0^{n-\lambda^T_j},1^{\lambda^T_j}).
$
This yields an ideal decomposition 
$
\,\Sss(\{\l\})=\prod_{j=1}^{\l_n} \,\Sss(\{(0^{n-\lambda^T_j},1^{\lambda^T_j})\})=\prod_{j=1}^{\l_n} \,I_{n,n-\l^T_j+1}\,
$
 by \Cref{prop:Borelsum}(1).   It remains to observe that, by definition of $\l^T$, the number of parts of $\l^T$ of size $n-i+1$, that is, the number of rows for the Young diagram of $\l$ which contain exactly $n-i+1$ boxes, is $\l_i-\l_{i-1}$. Thus combining the repeated factors of the previous identity yields the claim.
\end{proof}

\medskip
The notion of partition Borel generators has an interesting geometric application, as it allows to describe the Newton polytope of a symmetric strongly shifted ideal in terms of well-studied convex polytopes.
The {\em Newton polytope} of a monomial ideal $I$ is defined as the convex hull of the exponents of its minimal generators, that is,
\[
\np(I)=\conv\{(a_1,\ldots, a_n) \mid x^a\in G(I)\}.
\]
In convex geometry a {\em permutohedron} is a convex body defined as follows
\[
\bP(a_1,\ldots, a_n)=\conv\{(a_{\sigma(1)},\ldots, a_{\sigma(n)}) \mid \sigma\in \sym_n\},
\]
where $\conv$ denotes taking the convex hull of a set of points in $\mathbb{R}^n$. Permutohedra have emerged as objects of recent interest in combinatorics \cite{Postnikov} and in algebraic geometry \cite{Huh}. An important example of permutohedron is the {\em hypersimplex} $\Delta_{n,d}=\bP(0^{n-d}, 1^d)$.
We now show that the convex geometry of sssi's is governed by permutohedra. Some geometric implications of this fact will be discussed in  \Cref{section:Rees principalBorelsssi}.

\begin{prop}
\label{prop:permutohedron}
The Newton polytope of a principal Borel sssi is a permutohedron, namely,
\[
\np(\Sss(\{\l\}))=\bP(\l)=\conv\{(\sigma(\l) \mid \sigma\in \sym_n\}.
\]
In general, the Newton polytope of an sssi $I$ is a convex hull of a union of permutohedra:
\[
\np(I)=\conv\left(\bigcup_{\l\in B(I)} \bP(\l)\right).
\]
\end{prop}
\begin{proof}
It is clear from the definitions that $\bP(\l)\subseteq \np(\Sss(\{\l\}))$ since the vertices of the permutohedron are exponent vectors for some of the monomials in $G(\Sss(\{\l\})$. For the converse, a theorem by Rado \cite{Rado}, as transcribed in \cite[Proposition 2.5]{Postnikov}, states that $\bP(\l)$ is defined by the following (in)equalities
\[
\bP(\l)=\left \{(t_1,\ldots, t_n)\in \mathbb{R}^n \mid
\begin{cases}
 t_1+\cdots +t_n =\l_1+\cdots\l_n \\
 t_{i_k}+\cdots +t_{i_n}\leq \l_k+ \cdots +\l_n, & \forall 1\leq k\leq n, \forall 1\leq i_k\leq \ldots\leq i_n\leq n
\end{cases}
\right \}.
\]
Each exponent vector of a monomial in $G(I)$ satisfies the above system by \Cref{rem:dominance}, thus we obtain $\np(\Sss(\{\l\}))\subseteq \bP(\l)$.

To deduce the general statement from that regarding principal Borel sssi's it suffices to note that by \Cref{prop:Borelmeet}~(1) an arbitrary sssi $I$ decomposes as $I=\sum_{\l\in B(I)} \Sss(\{\l\})$ and that the Newton polytope of a sum of ideals is the convex hull of the union of the Newton polytopes of the summands.
\end{proof}

The Minkowski sum of polytopes $A,B$ is defined as $A+B=\{a+b \mid a\in A, b\in B\}$. As a corollary to our decomposition formula for principal Borel sssi's in \Cref{prop:Borel=product}, we recover a well-known decomposition for the permutohedron as a Minkowski sum of hypersimplices.

\begin{cor}
For $\l\in P_n$ and $\l_0=0$ there is an identity $\bP(\l)=\sum_{i=1}^{n}(\l_{i}-\l_{i-1})\Delta_{i,n}$.
\end{cor}
\begin{proof}
\Cref{prop:Borel=product} yields the following polyhedral identities as a consequence of the fact that the Newton polytope of a product of ideals is the Minkowski sum of the Newton polytopes of the summands
\[
\np(\Sss(\{\l\})=\sum_{i=1}^{n}\np( I_{n,i}^{\l_{i}-\l_{i-1}})=\sum_{i=1}^{n} (\l_{i}-\l_{i-1})\Delta_{i,n}.
\]
\end{proof}

\subsection{The polymatroidal structure of a principal Borel sssi}
\label{section:polymatroidal}

In this subsection, we discuss another combinatorial characterization of principal Borel sssi's which is a direct consequence of \Cref{prop:Borel=product}. We recall that a homogeneous ideal is termed {\em equigenerated} if all its minimal generators have the same degree. The following definition is due to Herzog and Hibi \cite[Definition 2.1 and Remark 6.4]{HHibi02}. 

\begin{defn}
\label{def:polymatroidal ideal}
An equigenerated monomial ideal $I$ is called a \emph{polymatroidal ideal} if any two monomial generators $x_1^{u_1} \cdots x_n^{u_n}$ and $x_1^{v_1} \cdots x_n^{v_n}$ of $I$ satisfy the following exchange property:
\begin{center}
For every $i$ so that $u_i > v_i$ there exists a $j$ so that $u_j < v_j$ and $(x_1^{u_1} \cdots x_n^{u_n})x_j/x_i \in I$.
\end{center} 
\end{defn}
The terminology refers to the fact that the exponent vectors $(u_1, \ldots u_n) \in \mathbb{Z}^{n}$ of the monomials  $x_1^{u_1} \cdots x_n^{u_n}$ generating a polymatroidal ideal form a set of bases of a \emph{discrete polymatroid}. 
The index $j$ in \Cref{def:polymatroidal ideal} can be chosen so that also $(x_1^{v_1} \cdots x_n^{v_n})x_i/x_j\in I$ \cite[Theorem 4.1]{HHibi02}. This is referred in the literature as the \emph{symmetric exchange property}, and is conjectured to determine the algebraic structure of the toric ring of a polymatroidal ideal \cite{{White},{HHibi02}} (see \Cref{conj:White} and our discussion therein).

\medskip

Examples of polymatroidal ideals include ordinary powers of square-free Veronese ideals and principal Borel ideals by \cite[Examples 2.6(c) and 9.4]{HHibi02}. Since powers of square-free Veronese ideals are principal Borel sssi's and principal Borel ideals become principal Borel sssi's after symmetrization in the sense of \Cref{thm:sshifted vs sstable}, it is then natural to ask whether \emph{every} principal Borel sssi is polymatroidal. The following theorem shows that this is indeed the case.

\begin{thm}
\label{thm:polymatroidal}
A principal Borel sssi is a polymatroidal ideal. In fact, a symmetric monomial ideal is polymatroidal if and only if it is a principal Borel sssi.
\end{thm}
\begin{proof}
That every principal Borel sssi is a polymatroidal ideal follows from \Cref{prop:Borel=product}, since square-free Veronese ideals 
are polymatroidal and products of polymatroidal ideals are polymatroidal by \cite[Theorem 5.3]{ConcaH}.

To prove the converse, we first show that every symmetric polymatroidal ideal is a sssi.
Let $I$ be a symmetric polymatroidal ideal. Let $\l \in \Lambda(I)$ be a partition with $\l_j<\l_i$ and let $\sigma=(i j) \in \sym_n$. Then, as $I$ is symmetric, the monomials $f=x^\l$ and $g=\sigma(f)$ are in $I$. For convenience, in the following we denote $\deg_i(h)$ the exponent of $x_i$ in a monomial $h$.
By assumption we have $\l_i=\deg_i(f)> \deg_i(g)=\l_j$ and $j$ is the only index for which $\l_j=\deg_j(f)< \deg_j(g)=\l_i$. Since $I$ is polymatroidal, the exchange property then yields that $f x_j/x_i  \in I$, so I is symmetric strongly shifted.

We next prove that every polymatroidal sssi must be a principal Borel sssi.
Let I be a polymatroidal sssi and suppose it is not principal Borel, thus there exist distinct $\l, \mu \in B(I)$. After possibly switching the names of $\l$ and $\mu$, we may assume that $\l_1=\mu_1, \l_2=\mu_2, \ldots, \l_{i-1}=\mu_{i-1}$ and $\l_i < \mu_i$.  
Since $I$ is polymatroidal, there exists an index $j$ so that $\l_j > \mu_j$ and $x^\mu x_j/x_i \in I$. Since $\l_k=\mu_k$ for $k<i$, and $\l_i<\mu_i$, $j$ must satisfy $i<j$.
Setting $\mu':= \pa(x^{\mu-e_i+e_j})$, the previous considerations yield that $\mu' \in P(I)$. Note that $\mu'=\mu-e_{i'}+e_{j'}$,  where if $\mu_i< \mu_j$ then $i'=\max\{k: \mu_k=\mu_i\}$ and $j'=\min\{k: \mu_k=\mu_j\}$ and if $\mu_i= \mu_j$ then $i'=\min\{k: \mu_k=\mu_j\}$ and $j'=\max\{k: \mu_k=\mu_j\}$ and in both cases  $i\leq i'<j'\leq j$ .
 Hence $x^\mu=x^{\mu'}x_{i'}/x_{j'}$ is obtained from $x^{\mu'}$ via a Borel move and consequently $\mu'\in \Lambda(I)$ is larger than $\mu$ in the dominance order. But this is a contradiction, since $\mu\in B(I)$ is a maximal element of $\Lambda(I)$ with respect to dominance by  \Cref{prop:Borelgens}~(2).
\end{proof}

\begin{rem}
It follows from work of Postnikov \cite{Postnikov} that a discrete polymatroid can be interpreted as the set of integer points of a {\em generalized permutohedron}. The latter is a convex polytope obtained by parallel translation of the facets of a permutohedron; see also \cite[Theorem 3.17]{CastilloLiu}. 
\Cref{prop:permutohedron} and \Cref{thm:polymatroidal} then imply that a symmetric discrete polymatroid is the set of integer points of a permutohedron.
\end{rem}

A polymatroidal ideal $I$ is said to satisfy the \emph{strong exchange property} if, for any two distinct monomial generators $x_1^{u_1} \cdots x_n^{u_n}$ and $x_1^{v_1} \cdots x_n^{v_n}$ of $I$ and all indices $i$ and $j$ so that $u_i > v_i$ and $u_j < v_j$, then $(x_1^{u_1} \cdots x_n^{u_n})x_j/x_i$ is in $I$ (see \cite[Definition 2.5]{HHibi02}).  

Notice that square-free Veronese ideals and their ordinary powers satisfy this property. However, this is not true for arbitrary principal Borel sssi's.

\begin{prop}
 \label{prop:SEP}
  Let $I=\Sss(\{ \l \})$ be a principal Borel symmetric strongly shifted ideal. Then $I$ satisfies the strong exchange property if and only if $\l$ is of one of the following types:
  \begin{enumerate}
    \item $\l= (a, \ldots, a)$ for some $a \neq 0 \in \N$;
    \item $\l= (a^s, b^{n-s})$ for some $a < b \in \N$, $s>0$;
    \item $\l= (a^s, b, c^{n-s-1})$ for some $a < b < c \in \N$, $s>0$.
  \end{enumerate}
\end{prop}
\begin{proof}
   Notice that $I$ satisfies the strong exchange property trivially if (1) holds. Let $\l$ be as in (2) and let $u=x_1^{u_1}\cdots x_n^{u_n}, v=x_1^{v_1}\cdots x_n^{v_n}$ be distinct monomial minimal generators of $I$. Then, if $u_i>v_i$ and $u_j<v_j$ it must be that $u_j=v_i=a$ and $u_i=v_j=b$. Then, for $u' = u x_j/x_i$, we have 
    \[ 
  \pa(u') = \begin{cases}
 (a^{s-1}, a+1, b-1, b^{n-s-1}) \unlhd \l & \text{ if } b > a+1,\\
  (a^{s-1}, b-1, a+1, b^{n-s-1}) = \l& \text{ if } b = a+1
  \end{cases}
 \]
   Thus, in any case $\pa(u') \in P(I)$, whence $u' \in I$. This shows that $I$ satisfies the strong exchange property.
   Similarly, if (3) holds, exchanging variables $x_i,x_j$ appearing with distinct exponents among distinct monomial generators $u, v$ produces $u' = u x_j/x_i$, where either $\pa(u')=\l$, or $\pa(u')= (a^{s-1}, a+1, b-1, c^{n-s-1})$, or $\pa(u')= (a^{s}, b+1, c-1, c^{n-s-2})$, or $\pa(u')= (a^{s-1}, a+1, b, c-1, c^{n-s-2})$. Since $\pa(u') \in P(I)$ in each of these cases, the strong exchange property is satisfied. 
   
   We next show that in all other cases the strong exchange property does not hold. Assume first that $\l=(a^s,b^t,c^q,d^r, \ldots)$ has at least four distinct parts $a<b<c<d$. Then there exist monomial generators of $I$ of the form $u=x_1^{a}x_2^{b}x_3^{c}x_4^{d}\cdots$ and $v=x_1^cx_2^ax_3^dx_4^b\cdots$ respectively. Hence, $u' = u x_3 /x_2$ is such that $\pa(u')= (a^s, b-1,b^{t-1}, c^{q}, c+1, d^r, \ldots)$. Since $\l \unlhd \pa(u')$, $\l \neq \pa(u')$ and $I = \Sss(\{ \l\})$, we deduce $\pa(u') \notin P(I)$, whence $u' \notin I$. 
   Finally, suppose that $\l=( a^s,b^t,c^q)$ with $a<b<c, s,q\geq 1, t\geq 2$ and consider minimal generators for $I$ of the form $u=x_1^ax_2^bx_3^bx_4^c\cdots, v=x_1^bx_2^ax_3^cx_4^b\cdots$. Then, $u' = u x_3 /x_2$ is such that $\pa(u')= (a^s, b-1, b^{t-2}, b+1, c^q, \ldots)$. Since $\l \unlhd \pa(u')$, it follows that $\pa(u') \notin P(I)$, which completes the proof.
\end{proof}

%
%

 The factorization of principal Borel sssi's given in \Cref{prop:Borel=product} parallels the known factorization of principal Borel ideals as products of monomial prime ideals; see, e.g., \cite[Propositions 2.7]{FMS13}. Polymatroidal ideals endowed with such a factorization property are called transversal polymatroidal ideals. In detail, a \emph{transversal polymatroidal ideal} is an ideal which can be written as a product of monomial ideals generated by subsets of the variables $x_1,\ldots, x_n$, with repeated factors allowed. 
The following proposition shows that principal Borel sssi's need not be transversal. In its statement, for $\l\in P_n$ we define the the discrete difference vectors $\Delta^i\l$ inductively by $\Delta^0 \l=\l$, $(\Delta \l)_j=\l_{j+1}-\l_j$, and $\Delta^i\l=\Delta(\Delta^{i-1}\l)$.

\begin{prop}
 \label{prop:transversal}
 The following are equivalent: 
  \begin{enumerate}
  \item $I=\Sss(\{ \l \})$ with $\l\in P_n$ is a transversal polymatroidal ideal,
  \item there exist integers $a_j\geq 0$ such that $\displaystyle{\l_i = \sum_{j=1}^{i} {i-1 \choose j-1} a_{j}}$ for $1 \leq i \leq n$,
  \item for some (equivalently, for each)  $1\leq i\leq n-1$ we have $\Delta^i\l \in P_{n-i}$ and $\Sss(\{\Delta^i\l\})$ is transversal polymatroidal.
  \item $(\Delta^i\l )_1\geq 0$ for all $0\leq i\leq n-1$.
  \end{enumerate}
\end{prop}
\begin{proof}
   Notice that a symmetric transversal polymatroidal ideal in $n$ variables is of the form
\[
J= \left( \prod_{i=1}^n (x_i)^{a_1} \right)  \left( \prod_{1\leq i<j\leq n} (x_i, x_j)^{a_2} \right)  \left( \prod_{1\leq i<j<k\leq n} (x_i, x_j,x_k)^{a_3} \right) \cdots (x_1, \ldots, x_n)^{a_n}
\]
for some integers $a_j\geq 0$. Indeed, if $(x_{i_1}, x_{i_2},\ldots, x_{i_c})^{a_c}$ is a factor in the product decomposition of $J$ then for each $\sigma\in\sym_n$ we have $(x_{\sigma(i_1)}, x_{\sigma(i_2)},\ldots, x_{\sigma(i_c)})^{a_c}$ as a factor in a product decomposition of $\sigma(J)=J$. Thus, it follows from the uniqueness of the decomposition of $J$ as a product of powers of monomial prime ideals \cite[Lemma 4.1]{HRV}  that $(x_{i_1}, x_{i_2},\ldots, x_{i_c})^{a_c}$ is a factor in the decomposition of $J$ if and only if for all $\sigma\in \sym_n$ the ideal $(x_{\sigma(i_1)}, x_{\sigma(i_2)},\ldots, x_{\sigma(i_c)})^{a_c}$ is a factor in the decomposition of $J$.

For the equivalence of (1) and (2), if $I=\Sss(\{\l\})$ is transversal then $I=J$ for some $J$ as described above. The partition Borel generator $\l$ of $I$ is the largest monomial in $G(I)$ with respect to the monomial order $\lex$ defined by $\alpha<_{\lex}\beta$ if the leftmost non-zero entry of $\alpha-\beta$ is positive,
since this order refines the dominance order on partitions. Thus, in order to 
establish whether $J=I$, we must identify the largest monomial $u$ in $G(J)$ with respect to this order. We claim that this monomial $u$ is obtained as follows: 
\begin{equation}
\label{eq:transversalgen}
u =  \left( \prod_{i=1}^n x_i^{a_1} \right) \left( \prod_{1\leq i<j\leq n} x_j^{a_2} \right) \left( \prod_{1\leq i<j<k\leq n} x_k^{a_3} \right)  \cdots x_n^{a_n}=\prod_{i=1}^n x_i^{\sum_{j=1}^{i} {i-1 \choose j-1} a_{j}}.
\end{equation}
Indeed, if $i_1<i_2<\cdots < i_c$, the largest monomial in $(x_{i_1}, x_{i_2},\ldots, x_{i_c})^{a_c}$ with respect to $\lex$ is $x_{i_c}^{a_c}$ and since this is a monomial order, hence compatible with products, it follows that $u$ is indeed the $\lex$-largest monomial of $G(J)$. The second equality in \eqref{eq:transversalgen} follows by observing that for fixed $c$ and  $i_c$ there are ${i-1 \choose j-1}$ ideals  $(x_{i_1}, x_{i_2},\ldots, x_{i_c})$ with $1\leq i_1<i_2<\cdots<i_c\leq n$.
Thus $\l=\pa(u)$ yields the desired description of the $\l_i$ in statement (2). Conversely, if the $\l_i$ can be described as in  statement (2) of our result, then we note that $J$ is a symmetric polymatroidal ideal and thus $J=\Sss(\pa(u))$ by \Cref{thm:polymatroidal}. Since $\pa(u)=\l$, we have $I=J$ and so $I$ is transversal.

Observe that the system of equations $\,\displaystyle{\l_i = \sum_{j=1}^{i} {i-1 \choose j-1} a_{j}}\,$ for $1 \leq i \leq n$ is equivalent to the system $\,\displaystyle{\l_{i+1}-\l_i = \sum_{j=2}^{i-1} {i-1 \choose j-2} a_{j}}\,$ for $1 \leq i \leq n$ in the unknowns $a_2, \ldots, a_n$, together with the equation $\l_1=a_1$. If the latter system admits nonnegative solutions then $\Delta\l$ is a partition, since the  functions  ${i-1 \choose j-2}$ are nondecreasing in the argument $i$. Moreover, this system admits nonnegative solutions if and only if $\Sss(\{\Delta\l\})$ is transversal polymatroidal, by the equivalence of (1) and (2). It follows by induction on $i$  that $(2)$ implies that, for each $1\leq i\leq n-1$,  $\Delta^i\l \in P_{n-i}$ and $\Sss(\{\Delta^i\l\})$ is transversal polymatroidal. Thus $(2)$ implies the stronger form of $(3)$ and hence also the weaker form.

The equivalence of the two systems of equations mentioned above shows that if $\Delta\l \in P_{n-1}$ and $\Sss(\{\Delta\l\})$ is transversal polymatroidal then $\Sss(\{\l\})$ is transversal polymatroidal. Thus, it follows by induction on $i$ that if for some $1\leq i\leq n$ we have $\Delta^i\l \in P_{n-i}$ and $\Sss(\{\Delta^i\l\})$ is transversal polymatroidal then $\Sss(\{\l\})$ is transversal polymatroidal. Thus the weaker (hence also the stronger) form of $(3)$  implies $(1)$.

Since  (3) clearly implies (4), we end by showing $(4)$ implies $(2)$. For this note the solution to the system of identities listed in $(2)$ is $a_i=(\Delta^{i}\l)_1\geq 0$. This can be seen by induction on $i$ using the fact that taking the difference of consecutive identities replaces the system in $(2)$ for $\l$ with the similar system for $\Delta\l$ and the equation $a_i=\l_1$, as noted above.
\end{proof}
 

  A special class of transversal polymatroidal ideals, which includes all principal Borel ideals, is that of \emph{lattice path polymatroidal ideals}. These ideals were introduced by Schweig in \cite{Schweig1} and are characterized by the property that their minimal generators correspond to certain planar lattice paths. The next corollary shows that lattice path polymatroidal ideals are rarely symmetric. In particular, the polymatroidal structure of principal Borel sssi's is usually different than that of principal Borel ideals.
 \begin{prop}
   Let $I$ be a lattice path polymatroidal ideal. $I$ is symmetric if and only if either $I = \Sss(\{ (a, ..., a) \})$ for some $a \in \N$, or $I$ is a power of the homogeneous maximal ideal.
 \end{prop} 
  \begin{proof}
    In \cite[Proposition 2.9]{Lu} it is shown that $I$ is a lattice path polymatroidal ideal if and only if $I$ is the product of monomial prime ideals, each of which is generated by consecutive variables: 
    \begin{equation}
      \label{eq:latticepath}
         I = (x_{s_1}, \ldots, x_{s_1+k_1}) \cdots (x_{s_n}, \ldots, x_{s_n+k_n}), 
      \end{equation}
     with $s_1 \leq \ldots \leq s_n$ and $s_1+k_1 \leq \ldots \leq s_n+k_n$. 
     In particular, for every $a \in \N$, the symmetric ideals $\Sss(\{ (a, ..., a) \})$ and $\m^a$ are lattice path polymatroidal ideals. 
     
     Conversely, if $I$ is a lattice path polymatroidal ideal, \eqref{eq:latticepath} implies that $I$ is transversal. If $I$ is also symmetric, then from the proof of \Cref{prop:transversal} it follows that, for every $i$, every permutation of $(x_{s_i}, ...., x_{s_i +k_i})$ must be a factor of $I$ too. 
     Thus, the only admissible factors are of the form $P_i = (x_i)$  for all $i$, or $\m=(x_1, \ldots, x_n)$. The first case yields  $I = \Sss \{ (a, ..., a) \}$ while the second case yields a power of the maximal ideal $\m^b$, where $a$ and $b$ respectively denote the multiplicity of each $P_i$ or of $\m$ in the factorization \eqref{eq:latticepath}.
  \end{proof}

\Cref{prop:SEP} and \Cref{prop:transversal} suggest that the combinatorial properties of principal Borel sssi's might be very different from those of polymatroidal ideals which satisfy the strong exchange property or are transversal. In fact, one can easily construct principal Borel sssi's which do not belong to any of these classes.
\begin{ex}
  \label{ex:notSEPnortransversal}
  Let $I= \Sss(\{ (1,3,4,5)\})$. By \Cref{prop:SEP}, $I$ does not satisfy the strong exchange property. Moreover, $I$ cannot be transversal, since the equations for the $a_j$'s in \Cref{prop:transversal} have no integer solutions for the given values $\l_1=1, \l_2=3, \l_3=4, \l_4=5$. Alternatively $\Delta\l=(2,1,1)$ is not a partition.
\end{ex}

Nevertheless, we will see in later sections that ordinary powers of principal Borel sssi's  share similar algebraic properties as powers of polymatroidal ideals that are either transversal or satisfy the strong exchange property. 

\subsection{Numerical invariants of symmetric shifted ideals}
\label{section:invariants}

Inspired by the case of stable and strongly stable ideals, in this subsection we use partition generators to calculate numerical invariants associated with symmetric shifted and strongly shifted ideals. In fact, most formulas will hold more generally for arbitrary symmetric ideals.

Denote $\min (a) = \min \{ i \colon a_i \neq 0 \}$, $\max (a) = \max \{ i \colon a_i \neq 0 \}$. If $I$ is strongly stable,  then
\[
\codim(I) = \max\{\min(a) \colon x^a\in G(I)\} \quad \mathrm{and} \quad \pd(R/I) = \max\{\max(a) \colon x^a\in G(I)\}
\]
The latter formula follows from the Eliahou-Kervaire resolution \cite{EK}. Moreover, in both cases it is enough to only consider the Borel generators of $I$; see \cite[Proposition 2.14]{FMS11}.

Replacing exponent vectors with partitions, we derive an analogous formula for the codimension of a symmetric strongly shifted ideal, which is actually satisfied by an arbitrary symmetric monomial ideal.  

\begin{prop}
\label{prop:height}
Let $I$ be a symmetric monomial ideal. For each $\l \in \L(I)$ denote $\min(\l)=\min\{i: \l_i>0\}$. 
The height of $I$ is given by 
\[
\codim(I)=\max_{\l\in \L(I)}\{ \min(\l)\}.
\]
Moreover, if $I$ is symmetric strongly shifted, then  $\codim(I)=\max_{\l\in B(I)}\{ \min(\l)\}$.
\end{prop}
\begin{proof}
From \Cref{prop:radical} it follows that $\sqrt{I}=I_{n,c}$, where $c=\codim(I)$.
It remains to determine the height of $I$. For this notice that a monomial $m\in \sqrt{I}$ if and only if $|\supp(m)|>n-c$. Therefore one has the chain of equalities
\begin{eqnarray*}
c &=&n-\min_{m\in I}\{|\supp(m)|\}+1=n-\min_{\l\in \L(I)}\{|\supp(\l)|\}+1 \\
&=&n-\min_{\l\in \L(I)}\{n-\min(\l)+1\}-1=\max_{\l\in \L(I)}\{\min(\l)\}.
\end{eqnarray*}
In the first line of the displayed equalities we use the fact that $\supp(m)=\supp(\pa(m))$. 

Suppose now that $I$ is a symmetric strongly shifted ideal. We prove that we can replace $\L(I)$ with $B(I)$ in the formula above by showing that each Borel move can only increase the size of the support.
This is because if $\mu_j>\mu_i$ and $\mu'=\mu-e_j+e_i$ then either $\mu_j-1>0$ and then $\supp(\mu)\subseteq \supp(\mu')$ or $\mu_j=1$ and $\mu_i=0$ in which case $\pa(\mu')=\mu$ and thus $|\supp(\mu')|=|\supp(\mu)|$.
\end{proof}

On the other hand, the projective dimension of a symmetric strongly shifted ideal cannot be expressed in terms of min/max of the partition Borel generators. 
%
%
%
In \Cref{prop:pd} below, we provide a formula for the projective dimension of a symmetric shifted ideal in terms of its partition generators using a different partition statistic we call $\med$. Recall from \cite[Theorem 3.1]{BDGMNORS} that the Betti numbers of a symmetric shifted ideal are given by
 \begin{equation}
    \label{eq:beta}
\beta_{i}(I)=\sum_{u\in G(I)}\binom{|G(C(u))|}{i},
\end{equation}
where for a monomial $u =\sigma(x^\lambda) \in G(I)$ with $\lambda \in P_n$ and $\sigma \in \sym_n$, the ideal $C(u)$ is constructed as follows (see \cite[proof of Theorem 3.2]{BDGMNORS}). 
First, one defines a total order on the set of monomials in $S=k[x_1,\ldots, x_n]$. 
\begin{defn}
  \label{defn:order}
  Let $\lambda,\mu \in \P_n$ and let $v= \tau(x^\mu)$ and $u=\sigma(x^\lambda)$ be distinct monomials in $S$, for some $\sigma, \tau \in \sym_n$. Define $v \prec u$ if one of the following conditions holds
\begin{enumerate}
\item $\mu <_\lex \lambda$, that is, either $|\mu| < |\l|$, or $|\mu| = |\l|$ and the leftmost non-zero entry of $\mu-\lambda$ is positive
\item $\mu=\lambda$ and $v <_\lex u$, that is, the leftmost non-zero entry of $\tau(\mu)-\sigma(\lambda)$ is positive.
\end{enumerate}
\end{defn}

Next, for a symmetric shifted ideal $I$ and a monomial $u =\sigma(x^\lambda) \in G(I)$, one defines $J=(v \in G(I): v \prec u)$ and
  \begin{equation}
    \label{eq:C(u)}
    C(u):=J:u=(x_{\sigma(1)},\dots,x_{\sigma(p)})
    +(x_{\sigma(k)}: p+1 \leq k \leq n-r,\ \sigma(k)<u_{\max}),
  \end{equation}
  where 
  \begin{equation*}
  \begin{split}
    &p=p(\lambda)=\#\{k: \lambda_k < \lambda_n-1\},\\
    &r=r(\lambda)=\#\{k: \lambda_k=\lambda_n\},\\
    &{u_{\max}=\max\{\sigma(k):\lambda_k=\lambda_n\}.}
  \end{split}
\end{equation*}

We are now ready to calculate the projective dimension of a symmetric shifted ideal. 
\begin{prop}
\label{prop:pd}
Let $I$ be a symmetric 
shifted ideal and define for every $\l\in P_n$ the integer $\med(\l)=|\{i: \l_i<\l_n\}|$. Then the projective dimension of $I$ is given by 
\[
\pd(R/I)=\max\limits_{\l\in \L(I)} \{\med(\l)\} +1.
\]
\end{prop}
\begin{proof}
Utilizing the notation in \eqref{eq:C(u)}, from the formula for the Betti numbers given by \eqref{eq:beta} we deduce that 
\[
\pd(I)=\max\{i: \exists u\in G(I) \text{ with } |G(C(u))|\geq i\}=\max\{|G(C(u))|: u\in G(I)\}.
\]
From \eqref{eq:C(u)} it follows that $|G(C(u))|\leq G(C(x^{\pa(u)}))$. Moreover, it can also be seen from the definitions that $|G(C(x^\l))|=\med(\l)$ for all $\l\in P_n$. Putting together the considerations above, we have shown that
\[
\pd(I)=\max\{\med(\l):\l\in \L(I)\}.
\]
\end{proof}

Unlike for height, when calculating the projective dimension of a symmetric strongly shifted ideal we cannot replace $\L(I)$ with the partition Borel generators of $I$.
\begin{ex}
  Let $I = \Sss(\{ (1,5,5) \})$. Notice that $(2,4,5) \in \L(I)$, however $\med((2,4,5))=2> 1= \med((1,5,5))$. 
\end{ex}

\medskip
When studying powers of an ideal $I$, a useful invariant is the \emph{analytic spread}, $\ell(I)$, of $I$ (i.e, the Krull dimension of the fiber cone $\mathcal{F}(I)$ of $I$, which we define in \Cref{section:Rees}). For instance, $\ell(I)$ controls the asymptotic growth of $\depth(R/I^k)$ as a function of $k$, thanks to Brodmann's formula \cite{Bro1}:
\[
  \lim_{k \to \infty} \depth(R/I^k) \leq n - \ell(I),
\] 
The analytic spread of an arbitrary ideal is usually difficult to calculate. However, if $I$ is an equigenerated monomial ideal, $\ell(I)$ coincides with the rank of the matrix whose rows are the exponent vectors of the monomial generators of $I$ (see \cite[Exercise 8.21]{HSwanson}). Using this fact, we can compute the analytic spread of any equigenerated symmetric monomial ideal.

%
 %
%

\begin{prop}
\label{prop:analyticspread}
  Let $I$ be an equigenerated symmetric monomial ideal. The analytic spread of $I$ is given by 
 \[ 
  \ell(I) = \begin{cases}
  n & \text{ if } I \neq \Sss(\{ (a, \ldots, a)\}) \text{ for any } a\in\N,\\
  1 & \text{ if } I = \Sss(\{ (a, \ldots, a)\}) \text{ for some } a\in\N.
  \end{cases}
 \]
\end{prop}
\begin{proof}
By \cite[Exercise 8.21]{HSwanson}, the analytic spread of an equigenerated monomial ideal $I$ is the rank of the matrix $M$ whose rows are the exponent vectors of the monomial generators of $I$. Hence, we only need to show that $M$ contains $n$ linearly independent rows. 

Let $V$ be the row space of $M$ viewed as a $\Q$-vector space. Since $I$ is a symmetric ideal, $V\subseteq \Q^n$ is a representation of $\sym_n$ acting naturally on $\Q^n$. Recall that the natural permutation representation of $\sym_n$ on $\Q^n$ decomposes into two irreducible representations: the trivial representation $T=\{(a,\ldots,a): a\in \Q\}$ and the standard representation $S=\{(q_1,\ldots, q_n): q_i\in \Q, \sum_{i=1}^n q_i=0\}$. Consider the projections $V_T$ and $V_S$ of $V$ onto $T$ and $S$ respectively. Since  $V_T$ and $V_S$ are  subrepresentations of the irreducible representations $T$ and $S$ respectively, we have $V_T=0$ or $V_T=T$ and $V_S=0$ or $V_S=S$ respectively, which yields four possibilities for $V$: $V=0$ or $V=T$ or $V=S$ or $V=S\oplus T=\Q^n$. 
Since $\Lambda(I)\subseteq V$, $\Lambda(I)$ contains at least one nonzero vector,  and no element of $\Lambda(I)$ is in $S$ (since every nonzero element of $S$ must have at least one negative coordinate), we are left with the possibilities $V=T$ or $V=S\oplus T=\Q^n$.

The case $V=T$ corresponds to $I = \Sss( \{ (a, \ldots, a)\})$ for some $a\in\N$, which gives $\ell(I)=\dim(V)=\dim(S)=1$. The case $V=\Q^n$ corresponds to 
$\Lambda(I) \cap S\neq \emptyset$ and gives $\ell(I)=\dim(V)=n$. 
\end{proof}

For $I\neq \Sss(\{ (a, \ldots, a) \})$, \Cref{prop:analyticspread} implies that $\ell(I) = \max_{\l\in \L(I)}\{\max(\l)\}$, where for each $\l\in \L(I)$ we denote $\max(\l)=\max\{i: \l_i>0\}$. A similar formula for the analytic spread of an equigenerated strongly stable ideal, with partitions replaced by exponent vectors of the monomial generators, was proved by Monta\~no in \cite[Proposition 5.2.6]{JonathanThesis}, using methods from convex geometry.

For a monomial ideal $I$, the {\em Newton polyhedron} of $I$ is defined as the Minkowski sum
\[
NP(I):= \np(I) + \mathbb{R}^n_{\geq 0}.
\] 
It is well known that the integral closure $\ov{I}$ of $I$ can be determined via the formula $NP(\ov{I})= NP(I) \cap \N^n$. Moreover, by \cite[Theorem 2.3]{BiviaAusina} the analytic spread of $I$ is 
\begin{equation}
\label{eq:compactfaces}
  \ell(I) = \max \{ \dim F \mid F \text{ is a compact face of } NP(I) \} +1.
\end{equation}
If $I$ is an equigenerated monomial ideal, the Newton polyhedron $NP(I)$ has a unique compact face of maximal dimension, which coincides with the Newton polytope $\np(I)$. 
As a consequence of \Cref{prop:analyticspread}, we then have the following description of the Newton polytope of an arbitrary  equigenerated symmetric monomial ideal.

\begin{cor}
\label{cor:uniquefacet}
 Let $I$ be an equigenerated symmetric monomial ideal and let $\np(I)$ denote the Newton polytope of $I$. Then
 \[
 \dim (\np(I)) = 
 \begin{cases}
 n-1 & \text{ if } I \neq \Sss(\{ (a, \ldots, a)\}) \text{ for any } a\in\N,\\
  0 & \text{ if } I = \Sss(\{ (a, \ldots, a)\}) \text{ for some } a\in\N.
 \end{cases}
 \]
\end{cor}
\begin{proof}
This is immediate from \eqref{eq:compactfaces} and \Cref{prop:analyticspread}. 
\end{proof}

In \cite{{HQureshi},{DHQureshi}}, the analytic spread of an equigenerated monomial ideal $I$ has been characterized 
in terms of the \emph{linear relation graph} $\Gamma = \Gamma(I)$ of $I$. This graph is defined as follows. If $G(I) = \{u_1, \ldots, u_m \}$ denotes a minimal monomial generating set for $I$, the edge set of $\Gamma$ is given by
\begin{equation}
\label{eq:Gamma}
E(\Gamma) = \{ \{ i, j \} : \text{ there exist } u_k, u_l \in G(I) \text{ such that } x_i u_k = x_j u_l   \},
\end{equation}
while the vertex set of $\Gamma$ is given by $\,\displaystyle{V(\Gamma) = \bigcup_{i : \{i,j\} \in E(\Gamma)}} \{i \}$. 

Assume further that $I$ is equigenerated and let $r$ and $s$ denote the number of vertices and the number of connected components of $\Gamma(I)$ respectively. By \cite[Lemma 4.2]{HQureshi} one has
\begin{equation}
  \label{eq:anspreadgraph}
\ell(I) \geq r -s +1,
\end{equation}
with equality holding if $I$ is a polymatroidal ideal, or, more generally,  if $I$ has linear syzygies \cite[Lemma 4.3]{DHQureshi}.
Since equigenerated symmetric shifted ideals have a linear resolution by \cite[Theorem 3.2]{BDGMNORS}, we then deduce the following immediate corollary.
\begin{cor}
   \label{cor:anspread shifted}
   Let $I$ be an equigenerated symmetric shifted ideal. Then, $\ell(I) = r -s +1$, where $r$ is the number of vertices in 
  $\Gamma(I)$ and $s$ is the number of its connected components.  
\end{cor}

In fact, we can give a full description of the linear relation graph of an equigenerated ssi, which we will use in \Cref{section:assprimes} to study the depths of powers of an ideal of this kind. 
\begin{prop}
   \label{prop:connected}
   Let $I$ be an equigenerated symmetric shifted ideal. If $I =\Sss(\{(a, \ldots, a)\})$ for some $a \in \N$, then the linear relation graph $\Gamma=\Gamma(I)$ is the graph with $n$ vertices and no edges. Otherwise, $\Gamma$ is connected, with $V(\Gamma) = \{ 1, \ldots, n \}$.
\end{prop}
\begin{proof}
The first claim is clear, since $I =\Sss(\{(a, \ldots, a)\}=(x_1^a \cdots x_n^a)$ is a principal ideal, so no  linear relations occur.

Suppose that $I \neq \Sss(\{(a, \ldots, a)\})$. Then, there exists a $\l \in \Lambda(I)$ with $\l_1 < \l_n$. Let $\mu := \l + e_1 - e_n$, then $\mu \in \Lambda(I)$ since $I$ is symmetric shifted. Moreover, $x_1 x^\l = x_n x^{\mu}$, that is, $\{ 1,n \} \in E(\Gamma)$. We next show that for any  $1 <i <n$ either $\{ i,n \} \in E(\Gamma)$ or $\{ 1,i \} \in E(\Gamma)$, whence  $\Gamma$ is a connected graph on $n$ vertices. 
If $\l_i < \l_n$, by the symmetric shifted property of $I$ we have that $\eta := \l + e_i - e_n \in \L(I)$. Moreover, $x_i x^\l = x_n x^{\eta}$, that is, $\{ i,n \} \in E(\Gamma)$. 
On the other hand, if $\l_i = \l_n$, let $\tau = (i,n) \in \sym_n$. Then, applying $\tau$ to the equality $x_1 x^\l = x_n x^{\mu}$ we get $x_1 x^\l = x_i \tau(x^{\mu})$. Since $\tau(x^{\mu}) \in G(I)$, this means that $\{ 1,i \} \in E(\Gamma)$ and the proof is complete.
\end{proof}

\section{Integrally closed sssi's}
\label{section:normality}

In this section and the next, we use the results from \Cref{section:Borelgens} to study ordinary powers of symmetric strongly shifted ideals. As it turns out, the symmetrization result of \Cref{thm:sshifted vs sstable} provides deep insight on this matter. 

We begin by identifying classes of symmetric strongly shifted ideals which are well-behaved with respect to integral closure.  
An ideal is said to be \emph{integrally closed} if it coincides with its integral closure, and is called \emph{normal} if all its powers are integrally closed. To understand normality of an ideal $I$, a useful object is the \emph{Rees ring} of $I$, which is defined as the subring $\R(I)=\bigoplus_{i\geq 0}I^i t^i \subseteq R[It]$. Indeed, since $R=K[x_1, \ldots, x_n]$ is a normal domain, $I$ is normal if and only if the Rees algebra $\R(I)$ is a normal domain \cite[Propositions 5.2.1 and 5.2.4]{HSwanson}. 

The following example shows that a sssi need not be integrally closed or normal.

\begin{ex}
  \label{ex:notintegrallyclosed}
   Consider $I=\Sss(\{(2,2,8), (0,6,6)\})$ and observe that for $\l=(1,4,7)$ one has $2\l=(2,2,8)+(0,6,6)$. Hence, $x^{2\l}\in I$ and $x^\l\in \ov{I}$. However $\l\not\in P(I)$ so $x^\l\not\in I$.
   Hence, $I$ is not integrally closed and therefore not normal. 
\end{ex}

However, \Cref{thm:sshifted vs sstable} indicates a strategy to construct sssi's which are integrally closed or normal.
In particular, we can identify two classes of normal symmetric strongly shifted ideals, obtained by symmetrization of certain normal strongly stable ideals. 

A monomial ideal is called a \emph{lex-segment ideal} if, for every degree $d$, the set of monomials of degree $d$ in $I$ forms a {\em lex-segment}; namely, for each degree $d$, there exists a monomial $u \in I$ of degree $d$ so that $I$ contains every degree-$d$ monomial $v$ which is larger than $u$ in the lexicographic order. It is well-known that lex-segment ideals are strongly stable, e.g. \cite[p.~103]{HHbook}.

\begin{prop}
\label{prop:normal}
   Let $J$ be an integrally closed strongly stable ideal, and let $\displaystyle{I = \bigcap_{\sigma \in \sym_n} \sigma(J)}$ be its symmetrization in the sense of \Cref{thm:sshifted vs sstable}. Then, $I$ is an integrally closed sssi. 
 
  Moreover, if $I$ is a sssi such that  $\bor(x^{B(I)})$  is normal, then $I$ is normal. In particular, $I$ is normal if either of the following conditions hold
\begin{enumerate}
  \item $B(I) = \{ \l \}$ for some $\l \in P_n$, i.e., $I$ is a principal Borel sssi; or
  \item $\bor(x^{B(I)})$ is an equigenerated lex-segment ideal.
\end{enumerate}
\end{prop}
\begin{proof}
   First, suppose that $J$ is integrally closed. Since every $\sigma \in \sym_n$ acts as an isomorphism on $R$, then $\sigma(J)$ is integrally closed for every $\sigma \in \sym_n$. As the intersection of integrally closed ideals is integrally closed \cite[Remark 1.1.3]{HSwanson}, it follows that $I$ is integrally closed.
   
   Now let  $I$ be a sssi  and $J=\bor(x^{B(I)})$. Then, for every $k \geq 1$, $J^k= \bor(x^{kB(I)})$ is strongly stable by \cite[Proposition 1.2]{GuoWu}. Moreover, by \Cref{cor:powers} and \Cref{prop:Borelsum} we have that $I^k$ is symmetric strongly shifted, with $I^k=\Sss(kB(I))$. Hence, \Cref{thm:sshifted vs sstable}~(2) implies that $\displaystyle{I^k = \bigcap_{\sigma \in \sym_n} \sigma(J^k)}$. 
   Therefore, from the first part of the proof it follows that $I^k$ is integrally closed whenever $J^k$ is, that is, $I$ is normal whenever $J$ is. 

Finally, to establish the two particular cases it remains to note that if $I$ is a principal Borel sssi then $J$ is a principal Borel ideal. Furthermore, principal Borel ideals and equigenerated lex-segment ideals are normal as 
 follows from \cite[Proposition 2.14]{DeNegri} (see also \cite[p.~3]{BrunsConca} for principal Borel ideals and \cite[Theorem 5.3.2]{JonathanThesis} for equigenerated lex-segment ideals).
\end{proof}


The normality of an ideal $I$ as in \Cref{prop:normal} provides information on the singularities of the Rees ring $\R(I)$ and of the \emph{special fiber ring} $\cF(I) := \R(I)/ \m \R(I)$. These rings naturally appear in algebraic geometry in the context of blow-up constructions, which are a fundamental tool in the study of singularities.   

\begin{prop}
\label{prop:normalRees}
Let $I=\Sss(B)$ be an equigenerated, normal sssi (e.g., $I$ satisfies one of the conditions in \Cref{prop:normal}). Then, the Rees ring $\R(I)$ and the special fiber ring $\cF(I)$ are Cohen-Macaulay normal domains.
Moreover, $\cF(I)$ has rational singularities if $K$ is of characteristic 0, and  is $F$-rational if $K$ is of positive characteristic
\end{prop}
\begin{proof}
Since $R$ is a normal domain and $I$ is a normal ideal, $\R(I)$ is normal by \cite[Propositions 5.2.1 and 5.2.4]{HSwanson}. Moreover, \cite[Corollary 6.2]{HHibi02} implies that $\cF(I)$ is also normal. Since $\cF(I) = [\R(I)]_{(0,*)}$ is a subring of $\R(I)$ and $\R(I)$ is a subring of the integral domain $R[t]$,  both $\R(I)$ and $\cF(I)$ are integral domains. Their Cohen-Macaulay property now follows from their normality, thanks to a well-known theorem of Hochster (see for instance \cite[Theorem B.6.2]{HHbook}).

By \cite[Proposition 1]{Hochster}, 
the normality of $\cF(I)$ is equivalent to $\cF(I)$ being a direct summand of a polynomial ring. Thus by \cite{Boutot} it follows that $\cF(I)$ has rational singularities if the characteristic of $K$ is zero. If $K$ has positive characteristic  $\cF(I)$ is strongly $F$-regular, and in particular is $F$-rational by \cite{HochsterHuneke}.
\end{proof}

In \Cref{section:Rees} we will investigate the algebraic structure of the Rees ring and special fiber ring of a symmetric strongly shifted ideal, in relation with the geometry of algebraic toric varieties. In particular, \Cref{cor:normalpolytope} will describe interesting geometric consequences of the normality property of $\R(I)$ for a principal Borel sssi $I$. 
In the rest of this section and in \Cref{section:assprimes}, we instead use \Cref{prop:normalRees} to analyze the asymptotic behavior of the ordinary powers of a sssi. 

An ideal $L \subseteq I$ is called a \emph{reduction} of $I$ if $\ov{L} = \ov{I}$; equivalently, if there exists an integer $r$ so that $LI^r = I^{r+1}$. If this is the case, it follows that $LI^k= I^{k+1}$ for every $k \geq r$. Hence, the reductions of an ideal $I$ give information on the growth of powers of $I$.
A reduction $L$ of $I$ is called a \emph{minimal reduction} if it is minimal with respect to inclusion. The smallest integer $r$ so that $LI^r = I^{r+1}$ for a minimal reduction $L$ of $I$ is called the reduction number of $I$ with respect to $L$, and is denoted by $r_L(I)$. The \emph{reduction number} of $I$ is 
\[
  r(I) = \min \{r_L(I) : L \text{ is a minimal reduction of } I \}.
\]
The reduction number of a normal symmetric strongly shifted ideal can be estimated as follows.
\begin{cor}
  Let $I \subseteq R=K[x_1, \ldots, x_n]$ be equigenerated, normal sssi (e.g., $I$ satisfies either one of the assumptions in \Cref{prop:normal}) and assume that $K$ is an infinite field.
  Then $r(I) \leq n-1$, and $r(I)=0$ if $I = \Sss(\{(a, \ldots, a)\})$ for some $a \in \N$.
\end{cor}
\begin{proof}
 By \Cref{prop:normalRees} $\R(I)$ is Cohen-Macaulay, whence \cite[Theorem 2.3]{JKatz} implies that $r(I) \leq \ell(I) -1$. The claim now follows from \Cref{prop:analyticspread}. 
\end{proof}


While the reductions of a monomial ideal need not be monomial, in \cite[Proposition 2.1]{Singla} Singla showed that every monomial ideal $I$ admits a unique {\em minimal monomial reduction}, that is, a reduction which is monomial and contains no other monomial reduction of $I$. We next determine the unique monomial reduction of a principal Borel sssi. 

\begin{prop}
\label{prop:minimalmonred}
For  $\l\in P_n$ define the monomial ideal generated by the $\sym_n$-orbit of $x^\l$ as 
\[
L(\l)=\left( \sigma(\l) \mid \sigma\in \sym_n \right).
\]
Then $L=L(\l)$ is the minimal monomial reduction of $\Sss(\{\l\})$ and we have $\Sss(\{\l\})=\overline{L}$. 
\end{prop}
\begin{proof}
Let $I=\Sss(\{\l\})$. In \cite[Proposition 2.1 and Remark 1.3]{Singla} it is shown that the exponents for the monomial generators of the minimal monomial reduction of $I$ correspond to the vertices of $NP(I)=\np(I) + \mathbb{R}^n_{\geq 0}$, which coincide with the vertices of $\np(I)$, as the vertices of a Minkowski sum are obtained as sums of vertices from each summand. 
Therefore, we only need to show that $V=\{\sigma(\l)\mid \sigma\in \sym_n\}$ is the set of vertices of $\np(I)$.
The latter form a subset of $V$, because $\np(I)=\bP(\l)$ by \Cref{prop:permutohedron}. Since $V$ forms an orbit under the $\sym_n$ action on $\mathbb{R}^n$ and since $\np(I)$ and hence its vertex set are  $\sym_n$-invariant, we conclude that $V$ coincides with the set of vertices of $\np(I)$. Hence, $L$ is indeed the minimal monomial reduction of $I$.
Now, the fact that $L$ is a reduction of $I$ implies that $\overline{L}=\overline{I}$, while $\overline{I}=I$ follows by \Cref{prop:normal}(1), finishing the proof.
\end{proof}


\section{Associated primes and primary decomposition}
\label{section:assprimes}


When studying the ordinary powers of an ideal $I$, one often wishes to understand how the depths and associated primes of $I^k$ depend on the exponent $k$. In turn, thanks to \eqref{eq:symbolicpowerMin} and \eqref{eq:symbolicpowerAss}, this process sometimes gives information on the symbolic powers of $I$.

For an ideal $I$ in a Noetherian ring, Brodmann \cite{{Bro1},{Bro2}} proved that, for $k \gg 0$, $\Ass(I^k) = \Ass(I^{k+1})$ and $\depth(R/I^k) = \depth(R/I^{k+1})$; also, $\displaystyle{\lim_{k \to \infty} \depth(R/I^k) \leq n - \ell(I)}$. 
The smallest number $k_0$ for which equality holds is called the \emph{index of stability} of $I$ and is denoted as $\astab(I)$; $\Ass(I^{k_0})$ is called the \emph{stable set of associated primes} and is denoted as $\Ass^{\infty}(I)$. Similarly, the smallest integer so that $\depth(R/I^k) =\depth(R/I^{k+1})$ is called the \emph{index of depth stability} of $I$, $\dstab(I)$. 

An ideal $I$ is said to satisfy the \emph{persistence property} if $\Ass(I^k) \subseteq \Ass(I^{k+1})$ for every $k \geq 1$. 

\smallskip

For monomial ideals in a polynomial ring, various combinatorial techniques have been used to study these notions by several authors; see, e.g., \cite{{HRV},{FMS13},{HVladoiu},{HQureshi},{KMafi}}. 
In \cite[Theorem 3.3]{HQureshi}, Herzog and Qureshi provide bounds for the depths of powers of an arbitrary equigenerated monomial ideal $I$ in terms of its linear relation graph $\Gamma = \Gamma(I)$; see \eqref{eq:Gamma} for the definition of this graph. 
%
If $r$ and $s$ denote the number of vertices and the number of connected components of $\Gamma$ respectively, they show that 
\begin{equation}
  \label{eq:depthgraph}
   \depth(R/I^k) \leq n - k - 1  \,\text{ for }\, 1 \leq k \leq r -s,
\end{equation}
as long as $r-s \geq 1$.
Using theirs result and our understanding of the linear relation graph of an equigenerated ssi from \Cref{prop:connected}, we next determine the depths of powers of ideals in this class. 
\begin{prop}
  \label{prop:depth shifted}
   Let $I$ be an equigenerated symmetric shifted ideal of a polynomial ring in $n$ variables. Then, $\depth(R/I^k) \geq \depth(R/I^{k+1})$ for all $k \geq 1$. Moreover
   \begin{enumerate}
   \item if $I = \Sss(\{ (a, \ldots, a)\})$ for some $a \in \N$, then $\depth(R/I^k)=n-1$ for all $k\geq 1$, and $\dstab(I)=1${\rm ;}
   \item otherwise, $\depth(R/I^k) \leq n - k - 1$ for $1 \leq k \leq n-1$ and $\depth(R/I^k)= 0$ for $k \geq n-1$; in particular, $\dstab(I) \leq n-1$ and $\m \in \Ass(I^k)$ for every $k \geq n-1$.
\end{enumerate}
\end{prop}
\begin{proof}
Since equigenerated symmetric shifted ideals have a linear free resolution by \cite[Theorem 3.2]{BDGMNORS}, it follows from \cite[Proposition 2.2]{HRV} that 
\begin{equation}
 \label{eq:decreasingdepth}
   \depth(R/I^k) \geq \depth(R/I^{k+1}) \text{ for all } k \geq 1.  
\end{equation}

If $I = \Sss(\{ (a, \ldots, a)\})$ for some $a \in \N$, then for all $k\geq 1$ one has $I = \Sss(\{ (ka, \ldots, ka)\})$ by \Cref{prop:Borelsum}(2). Hence,  for all $k\geq 1$, $\depth(R/I^k)= n - \pd(R/I^k) = n-1$, where the latter equality follows from \Cref{prop:pd}. This implies that $\dstab(I)=1$.

Assume now that $I \neq \Sss(\{ (a, \ldots, a)\})$ for any $a \in \N$. Then, by \Cref{prop:connected} the linear relation graph $\Gamma$ of $I$ is connected, with $V(\Gamma) = \{ 1, \ldots, n \}$. Thus, \eqref{eq:depthgraph} implies that $\depth(R/I^k) \leq n - k - 1$ for $1 \leq k \leq n-1$. In particular, $\depth(R/I^{n-1})=0$, whence from \eqref{eq:decreasingdepth} we obtain that $\depth(R/I^k)=0$ for all $k \geq n-1$. The remaining claims now follow immediately (see also \cite[Corollary 3.4]{HQureshi}).
\end{proof}

With a different method, in \cite[Theorem 1.3]{HQureshi} Herzog and Qureshi also prove that an equigenerated graded ideal $I$ satisfies the persistence property if $I^{k+1} : I = I^k$ for all $k \geq 1$. By work of Ratliff \cite[Proposition 4.7]{Ratliff} the latter condition is satisfied whenever $I$ is a normal ideal.   \Cref{prop:normal} then implies the following result for symmetric strongly shifted ideals that either have exactly one partition Borel generators, or are the symmetrization of a lex-segment ideal in the sense of \Cref{thm:sshifted vs sstable}.

\begin{prop}
   \label{prop:persistence}
     Let $I=\Sss(B)$ be an equigenerated, normal symmetric strongly shifted ideal (e.g. $I$ satisfies one of the assumptions in \Cref{prop:normal}). Then 
     \begin{enumerate}
        \item For every $k \geq 1$, $I^{k+1} : I = I^k$.
        \item For every $k \geq 1$, $\Ass(I^k) \subseteq \Ass(I^{k+1})$ for every $k \geq 1$. That is, $I$ satisfies the persistence property.
       \item $\displaystyle{\lim_{k \to \infty} \depth(R/I^k) = n - \ell(I) = 
     \begin{cases}
     0 & \text{ if } \l\neq  (a, \ldots, a) \text{ for any } a\in\N,\\
     n-1 & \text{ if } \l= (a, \ldots, a) \text{ for some } a\in\N.
    \end{cases}}$
    \end{enumerate}
\end{prop}
 \begin{proof}
   Since $I$ is a normal ideal, \cite[Proposition 4.7]{Ratliff} implies (1), whence \cite[Theorem 1.3]{HQureshi} follows from (2). 
   Moreover, since the Rees ring $\R(I)$ is Cohen-Macaulay by \Cref{prop:normalRees}, a well-known result of Huneke implies that equality holds in Brodmann's inequality $\displaystyle{\lim_{k \to \infty} \depth(R/I^k) \leq n - \ell(I)}$ (see, e.g., \cite[Proposition 10.3.2]{HHbook}). 
  This implies the first equality in (3), whence the second equality follows from \Cref{prop:analyticspread}.
 \end{proof}
 

 While the techniques used to prove \cite[Theorem 1.3]{HQureshi} do not seem to apply to arbitrary equigenerated symmetric strongly shifted ideals, we currently do not know of any sssi failing to satisfy \Cref{prop:persistence}(1). This motivates the following questions. 

\begin{quest}
 Is it true that $I^{k+1} : I = I^k$ for every $k \geq 1$ if $I$ is an arbitrary symmetric strongly shifted ideal? Does any sssi satisfy the persistence property?
\end{quest}

\subsection{Stable associated primes of a principal Borel sssi}
\label{section:stableAss}

When $I$ is a principal Borel sssi, its polymatroidal nature complements the information provided by \Cref{prop:depth shifted} and \Cref{prop:persistence}, allowing to estimate the indices of stability and depth stability of $I$.

 \begin{prop}
   \label{prop:astab polymatroidal}
   Let $I=\Sss(\{\l\})$ be a principal Borel symmetric strongly shifted ideal of a polynomial ring in $n$ variables. Then 
   \begin{enumerate}
      \item $\dstab(I)=1$ if and only if either $\l= (a, \ldots, a)$ for some $a\in\N$,  or $\l\neq  (a, \ldots, a)$ for any $a\in\N$ and $\m=(x_1,\ldots,x_n) \in \Ass(I)$.
      \item $\dstab(I) \leq \astab(I) \leq n-1$, provided $n\geq 2$.
   \end{enumerate}
\end{prop}
\begin{proof}
From \Cref{prop:depth shifted} it follows that, for all $k\geq 1$,
\begin{equation}
\label{eq:depthsequence}
  \depth(R/I) \geq \depth(R/I^k) \geq \lim_{k \to \infty} \depth(R/I^k). 
\end{equation}
Then,  $\dstab(I)=1$ if and only if $\depth(R/I) = \lim_{k \to \infty} \depth(R/I^k)$.
By \Cref{prop:persistence}(3), the latter equality is equivalent to the condition that $\m \in \Ass(I)$ if $\l\neq  (a, \ldots, a)$ for any $a\in\N$, and 
follows from \Cref{prop:depth shifted}(1) otherwise. 
This proves (1).

To prove (2), assume first that $\l\neq (a, \ldots, a)$ for any $a\in\N$. Then, from the persistence property and \eqref{eq:depthsequence} it follows that
\[
 \dstab(I) = \min \{ k : \depth(R/I^k) =0 \} =  \min \{ k : \m \in \Ass(I^k)  \}.
\]
Hence, \Cref{prop:depth shifted}(2) implies that $\dstab(I) \leq \astab(I) \leq n-1$ (see also \cite[Lemma 2.20]{KMafi} and \cite[Theorem 4.1]{HQureshi}).
Finally, if $\l = (a, \ldots, a)$ for some $a\in\N$, then $I$ is a transversal polymatroidal ideal, whence $\astab(I)=\dstab(I) =1 \leq n-1$ by \cite[Corollaries 4.6 and 4.14]{HRV}.
\end{proof}

The inequality $\dstab(I) \leq \astab(I)$ in \Cref{prop:astab polymatroidal}(2) is remarkable, as for an arbitrary monomial ideal $I$ either of the integers $\astab(I)$ and $\dstab(I)$ might be smaller than the other; see \cite[p.~295]{HRV}. 
Moreover, it is known that $\dstab(I)=\astab(I)$ if $I$ is a transversal polymatroidal ideal \cite[Corollaries 4.6 and 4.14]{HRV}, an ideal of Veronese type \cite[Corollary 5.7]{HRV}, or a polymatroidal ideal with the strong exchange property \cite[Proposition 2.15]{KMafi}. Since these classes include principal Borel ideals and ordinary powers of square-free Veronese ideals, it then makes sense to ask whether $\dstab(I) = \astab(I)$  for all principal Borel symmetric strongly shifted ideals. 

Although this is not true in general (see \Cref{ex:astabNOTdstab} below), in this section we prove that several principal Borel sssi's satisfy $\dstab(I)=\astab(I) =1$. In fact, in \Cref{thm:stableAss} we determine the stable set of associated primes for every principal Borel sssi. As a preliminary result, we prove that every principal Borel sssi can be decomposed as a finite intersection of symbolic powers of square-free sssi's.

\begin{prop}
\label{prop:intersection symbolic}
   Let $\l = (\l_1, \ldots, \l_n)$ be a partition. For $1 \leq j \leq n$, let $P_j = (x_1, \ldots, x_j)$ and denote $a_j =\sum_{i=1}^{j} \l_i$.  Then, the principal Borel sssi $\Sss(\{ \l \})$ can be decomposed as
   \[
  \Sss(\{ \l \}) = \bigcap_{j=1}^n  \left(  \bigcap_{\sigma \in \sym_n} \sigma (P_j)^{a_j} \right) = \bigcap_{j=1}^n I_{n,j}^{(a_j)}.
   \]
\end{prop}
\begin{proof}
By \Cref{thm:sshifted vs sstable} we know that $\displaystyle{I =  \bigcap_{\sigma \in \sym_n} \sigma(\bor(x^\l))}$. Moreover, from \cite[Proposition 2.7 and Theorem 3.1]{FMS13} it follows that
  \[
 \bor(x^\l) = \prod_{j=1}^n P_j^{\l_j} = \bigcap_{j=1}^n P_j^{a_j},
  \]
 where $\displaystyle{a_j = \sum_{i : P_i \subseteq P_j} \l_i}=\sum_{i=1}^{j} \l_i$. 
  Therefore, we deduce the identity
 \[
 \Sss(\{ \l \}) =  \bigcap_{\sigma \in \sym_n} \sigma(\bor(x^\l)) =  \bigcap_{\sigma \in \sym_n} \sigma \left( \bigcap_{j=1}^n P_j^{a_j} \right) = \bigcap_{j=1}^n  \left(  \bigcap_{\sigma \in \sym_n} \sigma (P_j)^{a_j} \right).
 \] 
 For every $j$ we can rewrite $ \displaystyle{\bigcap_{\sigma \in \sym_n} \sigma (P_j)^{a_j} = I_{n,j}^{(a_j)}}$, which proves the last equality.
\end{proof}

Notice that the first equality in \Cref{prop:intersection symbolic} gives a (possibly redundant) primary decomposition of $I = \Sss(\{ \l \})$, since for every $j$ and every $\sigma \in \sym_n$ the ideals $\sigma (P_j)^{a_j}$ are $P_j$-primary. 
Since the powers of a principal Borel sssi are still principal Borel sssi's, \Cref{prop:Borelsum} and \Cref{prop:intersection symbolic} together imply that, for every integer $k \geq 1$,
\begin{equation}
   \label{eq:primary}
   I^k = \Sss( \{ \l \} ) ^k = \Sss( \{ k\l \} ) = \bigcap_{j=1}^n I_{n,j}^{(ka_j)}.
\end{equation}
In particular, as detailed in the proof of \Cref{thm:irredundant} below, for every $k \geq 1$, we have $\Ass(I^k) \subseteq \{\sigma(P_1), \ldots, \sigma(P_n) : \sigma \in \sym_n \}$, and $\sigma (P_j) \in \Ass(I^k)$ if and only if $I_{n,j}^{(ka_j)}$ is needed in the decomposition given by \eqref{eq:primary}.

\begin{ex}
\label{ex:astabNOTdstab}
Let $I=\Sss(\{1,2,2,4,4\})$. Irredundant primary decompositions of $I$ and $I^2$ computed using Macaulay2 \cite{M2} are the following:
\[ I = I_{5,1}\cap I_{5,2}^{(3)} \cap I_{5,4}^{(9)} \cap I_{5,5}^{(13)}, \] 
\[ I^2 = I_{5,1}^{(2)}\cap I_{5,2}^{(6)} \cap I_{5,3}^{(10)} \cap I_{5,4}^{(18)} \cap I_{5,5}^{(26)}. \]
Thus $\astab(I)=2$ follows from  \eqref{eq:primary} and the persistence property \Cref{prop:persistence}(1). Moreover, $\dstab(I)=1$ by \Cref{prop:persistence}(3), since $\m = I_{5,5} \in \Ass(I)$.
\end{ex}

To determine $\Ass^{\infty}(I)$, $\astab(I)$ and $\dstab(I)$ for an arbitrary principal Borel sssi, one needs to be able to predict which components in \eqref{eq:primary} are irredundant in a systematic way. We identify sufficient conditions in \Cref{thm:irredundant} below. Our proof relies on the following fact from 
 \cite[Proposition  4.1]{BDGMNORS}): for all $j \geq 1$,
\begin{subequations}
 \begin{align} 
  P \left( I_{n,j}^{(ka_j)} \right) & =   \{ \mu \in P_n : \sum_{i=1}^j \mu_i \geq k a_j \},  \label{eq:symbolicpart}\\ 
  \Lambda \left( I_{n,j}^{(ka_j)} \right) & = \{ \mu \in P_n : \sum_{i=1}^j \mu_i = k a_j \text{ and } \mu_i = \mu_j  \text{ for all }  i > j\}. \label{eq:symbolicgen}
  \end{align}
\end{subequations}

\begin{thm}
\label{thm:irredundant}
  Let $I = \Sss(\{ \l \})$ and $j' :=\min(\l)$. Adopt the notation of \Cref{prop:intersection symbolic}.  Then, for a fixed $k \geq 1$,
  \begin{enumerate}[leftmargin=5.5mm]
  \item all components $I_{n,j}^{(ka_j)}$ with $j<j'$ are redundant in \eqref{eq:primary}{\rm;} 
  \item  $I_{n,j'}^{(ka_{j'})}$ is not redundant in \eqref{eq:primary} and moreover $I^{(k)_\Min}=I_{n,j'}^{(k\l_{j'})}${\rm;} 
  \item if $j> j'$ and $\l_{j-1}<\l_j$ then the component $I_{n,j}^{(ka_j)}$ is not redundant in \eqref{eq:primary}{\rm;}        
  \item if $\l_1<\l_j$ and either $k>j(j-1)$ or $\displaystyle{\l_j>\frac{q+j-r}{k}}$ where $k \sum_{i=1}^{j-1}\l_i = (j-1)q+r$ and $0\leq r\leq j-2$,
      then the component $I_{n,j}^{(ka_j)}$ is not redundant in \eqref{eq:primary} if and only if $j\geq j'$.
        \item  if $\l_1=\l_j$ for some $j>1$ then the component $I_{n,j}^{(ka_j)}$ is redundant in \eqref{eq:primary}.
  \end{enumerate}
 
 In particular, for $k \gg 0$ a minimal primary decomposition of $I^k$ is
   \begin{equation}
   \label{eq:primarydecomp}
   I^k=\bigcap_{j \in \mathcal{J}} \bigcap_{\sigma\in \sym_n} \sigma(P_j)^{ka_j},
\end{equation}
   where
   \[
 \mathcal{J}=
\begin{cases}
       \{ j : j \geq j' \}, & \text{ if } j' >1\\
      \{1 \} \cup \{ j : j >1 \text{ and } \l_j\neq \l_1 \} , & \text{ if } j'=1.
      \end{cases}
 \]
  \end{thm}
\begin{proof}
For simplicity of notation, in the rest of the proof we denote $\sum_{s} \l := \sum_{r=1}^{s} \l_r$ for any partition $\l \in P_n$. 

(1)  For $j< j'$ one has $a_j=0$, so the corresponding component $I_{n,j}^{(ka_j)}=R$ of the decomposition \eqref{eq:primary} is redundant. 
  
 (2) Since by \Cref{prop:height} one has $\codim(I)=\min(\l) =j'$, the component $I_{n,j'}^{(ka_{j'})}$ is the intersection of the primary ideals $\sigma(P_{j'}^{ka_j'})$ which belong to minimal primes of $I^k$, that is, $I^{(k)_\Min}=I_{n,j'}^{(ka_{j'})}$. It remains to note that $a_{j'}=\l_{j'}$ by definition.

(3) We now establish irredundancy of the components $I_{n,j}^{(a_j)}$ so that $j\geq j'$ and $\l_{j-1}<\l_j$. This amounts to showing that $\bigcap_{i\neq j} I_{n,i}^{(ka_i)} \not \subseteq I_{n,j}^{(ka_j)}$.
 Equation \eqref{eq:symbolicpart} implies that for $\mu\in P_n$, $ x^\mu \in  I_{n,i}^{(ka_i)} $ if and only if $\sum_i\mu \geq k a_i=k\sum_i \l$. Set 
 \[
 \mu_i=
 \begin{cases}
 k\l_i & \text{ for } 1\leq i<j\\
 k\l_{j-1} & \text{ for } i=j\\
k (\l_{j+1}+\l_j-\l_{j-1})& \text{ for } i=j+1, \text{ provided }j+1\leq n\\
 k\l_i & \text{ for } j+2\leq i\leq n.
 \end{cases}
 \]
 Then $\mu\in P_n$ and $x^\mu\in I_{n,i}^{(ka_i)}$ for each $i\neq j$ since $\sum_i \mu=k\sum_i\l$. By contrast,  we have  $\sum_j\mu<k\sum_j \l $ since $\l_{j-1}<\l_j$ and thus $x^\mu\not \in I_{n,j}^{(ka_j)}$. We conclude that  the component $I_{n,j}^{(ka_j)}$ is not redundant in the decomposition \eqref{eq:primary}.
 
 (4) Now suppose $j > j'$ and in particular $j>1$. We construct $\mu\in P_n$ so that $x^\mu\in \bigcap_{i\neq j} I_{n,i}^{(ka_i)}$.
 By \eqref{eq:symbolicpart}, it must satisfy $\sum_i\mu\geq k\sum_i \l$ for each $i\neq j$. Set 
  \begin{equation}
 \label{eq:j-1}
 k\sum_{j-1} \l=(j-1)q+r \text{\, with \,}0\leq r\leq j-2
 \end{equation}
  and consider the partition given by 
 \begin{equation}
 \label{eq:mu}
 \mu_i=
 \begin{cases}
  q+1& \text{ for } 1\leq i\leq j \\
 N&\text{ for } j+1\leq i\leq n,
 \end{cases}
\end{equation}
where $N\gg0$ is a sufficiently large  integer such that $N\geq q+1$ and  $j(q+1)+N(i-j)\geq k\sum_i\l$ for each $i>j$. 
 By construction this satisfies $\sum_i\mu\geq k\sum_i\l$ for  $i>j$.
%
 If $p\in P_n$ and $a,b\in\N$ with $a\leq b$, then
  \begin{equation}
  \label{eq:ij}
 \sum_{b}p\geq \sum_a p + (b-a)p_a\geq  \sum_a p +\frac{b-a}{a} \sum_a p=\frac{b}{a}\sum_a p.
\end{equation}

 For the $\mu$ defined in \eqref{eq:mu} and for $i<j$ we have
 
  \begin{eqnarray*}
 \sum_i \mu &= & i(q+1) \\
 &= & i\left[\frac{k\sum_{j-1}\l-r}{j-1} + 1\right] \qquad \text{by } \eqref{eq:j-1}\\
 &\geq& k \sum_{i} \l+\frac{i(j-r-1)}{j-1} \qquad \text{by } \eqref{eq:ij} (a=i, b=j-1, p= \l)\\
 &\geq& k \sum_{i} \l  \qquad \text{since } j-r-1\geq 0.
 \end{eqnarray*}
Similarly we have
 \begin{eqnarray*}
 \label{eq:muj}
  \sum_j \mu &=& j(q+1) =\frac{jk\sum_{j-1}\l}{j-1}+ \frac{j(j-r-1)}{j-1}\\
 &=& k\sum_{j-1}\l+k\left[\frac{\sum_{j-1}\l}{j-1}+ \frac{j(j-r-1)}{k(j-1)} \right],
\end{eqnarray*} 
whence 
$ \sum_j \mu \geq k\sum_j\l$ if and only if 
\begin{equation}
\label{eq:asscondition}
\l_j \leq \frac{\sum_{j-1}\l}{j-1}+ \frac{j(j-r-1)}{k(j-1)}
=\frac{q+j-r}{k}.
\end{equation}
Thus, $x^\mu\in I_{n,j}^{(ka_j)}$ if and only if \eqref{eq:asscondition} is satisfied. 

Whenever the inequality  \eqref{eq:asscondition} fails,  we have $\bigcap_{i\neq j} I_{n,i}^{(ka_i)}\not \subseteq I_{n,j}^{(ka_j)}$ and hence  $I_{n,j}^{(ka_j)}$ is needed in \eqref{eq:primary}. 
Assume $k>j(j-1)$. Then we see that     
 \[
\left \lfloor \frac{\sum_{j-1}\l}{j-1}+ \frac{j(j-r-1)}{k(j-1)}\right \rfloor  \leq \frac{\sum_{j-1}\l}{j-1}.
  \]
If \eqref{eq:asscondition} holds, it implies $\l_j\leq \frac{\sum_{j-1}\l}{j-1}$. 
  But this is possible if and only if $\l_1=\cdots=\l_{j-1}=\l_j$, in which case equality holds in \eqref{eq:asscondition}. Since by assumption $\l_1< \l_j$, \eqref{eq:asscondition} must then fail, so $I_{n,j}^{(ka_j)}$ is not redundant. 
    
(5) If $\l_1 = \l_j$ for some $j>1$, then $\l_1=\cdots=\l_{j-1}=\l_j$, whence $ \l_j = \frac{\sum_{j-1}\l}{j-1}$. Let $\mu \in P_n$ be defined as in \eqref{eq:mu}.
If $x^\mu\in  I_{n,j-1}^{(ka_{j-1})}$, \eqref{eq:ij} with $a=j-1,\, b=j, \,p=\mu$ then yields
 \[
  \sum_j \mu\geq \frac{j}{j-1}\sum_{j-1}\mu \geq \frac{j}{j-1} k\sum_{j-1}\l \geq k\sum_j \l.
 \] 
 This shows that $ I_{n,j-1}^{(ka_{j-1})}\subseteq  I_{n,j}^{(a_j)}$ and thus the latter ideal is redundant in \eqref{eq:primary}.
 
 The formula for the primary decomposition follows by substituting $I_{n,j}^{(ka_j)}=\bigcap_{\sigma\in \sym_n} \sigma(P_j)^{ka_j}$  in \eqref{eq:primary} and removing redundant components. In detail, $\sigma(P_j)^{ka_j}$ is irredundant in 
 \begin{equation}
 \label{eq:primarynonmin}
 I^k=\bigcap_{j=1}^n \bigcap_{\sigma\in \sym_n} \sigma(P_j)^{ka_j}
\end{equation}
   if and only if $\sigma(P_j)\in \Ass(I^k)$. As $I^k$ is symmetric, $P_j\in \Ass(I^k)$ if and only if $\sigma(P_j)\in \Ass(I^k)$ for all $\sigma\in \sym_n$, so we see that the primary components $\sigma(P_j)^{ka_j}$ are either all simultaneously redundant in \eqref{eq:primarynonmin}, in which case $I_{n,j}^{(ka_j)}$ is redundant in \eqref{eq:primary}, or simultaneously irredundant in \eqref{eq:primarynonmin}, in which case $I_{n,j}^{(ka_j)}$ is iredundant in \eqref{eq:primary}.
 Formula \eqref{eq:primarydecomp} follows from \eqref{eq:primarynonmin} by means of statements (1)--(5) and the previous considerations provided $k > j (j-1)$ is satisfied for every $j$. 
 \end{proof}

Combining the previous results together, we can determine the stable set of associated primes of a principal Borel sssi.
\begin{thm}
\label{thm:stableAss}
   Let $I = \Sss(\{ \l \})$. Adopt the notation of \Cref{prop:intersection symbolic} where $j' = \min(\l)$.  Then
    \[
    \Ass^{\infty}(I) = 
    \begin{cases}
     \{\sigma(P_{j}): j'\leq j\leq n, \, \sigma \in \sym_n \} , & \text{ if } j'>1,\\
      \{\sigma(P_{j}) : j=1\text{ or } (j>1 \text{ and } \l_j\neq \l_1),  \, \sigma \in \sym_n \} , & \text{ if } j'=1.
      \end{cases}
    \]
    Moreover, 
    \begin{enumerate}
    \item if either $\l = (a, \ldots, a)$ for some $a \in \N$, or $\l\in P_n$ has no repeated parts other than possibly allowing for repetitions of $\l_1$, then $\astab(I)=\dstab(I)=1$ and for each $k\in \N$, $I^k= I^{(k)_{\Ass}}${\rm;}
    \item otherwise, setting $s=\max\{j:\l_1<\l_{j-1}=\l_j\}$ we have $\dstab(I) \leq \astab(I) \leq \min \{ n-1, s(s-1)+1 \}$.
    \end{enumerate}
\end{thm}

\begin{proof}
The claim regarding the stable set of associated primes follows from the primary decomposition  \eqref{eq:primarydecomp} in \Cref{thm:irredundant}.

For (1), from \Cref{thm:irredundant} (3) and (5) it follows that, in either cases, $\Ass(I) = \Ass(I^k)$ for every $k \geq 1$, whence $\astab(I)=1$. Moreover, 
$\dstab(I)=1$ by \Cref{prop:astab polymatroidal}(1), while the claim about the symbolic powers $I^{(k)_{\Ass}}$ follows from \eqref{eq:saturationAss}.  
%

For (2), notice that $\m \in \Ass(I^k)$ for all $k>s(s-1)$ by \Cref{thm:irredundant} (3) and (4). Hence, the conclusion follows from the persistence property and \Cref{prop:astab polymatroidal}(2).
\end{proof}

\begin{rem}
From \Cref{thm:stableAss} it follows that $\astab(I)=\dstab(I)=1$ for many principal Borel sssi's. While this equality is known to hold for transversal polymatroidal ideals by \cite[Corollaries 4.6 and 4.14]{HRV}, \Cref{ex:notSEPnortransversal} shows that there exist principal Borel sssi's with $\astab(I)=\dstab(I)=1$ which are not transversal. Hence, \Cref{thm:stableAss} produces examples of monomial ideals with $\astab(I)=\dstab(I)=1$ which were not previously included in \cite[Corollaries 4.6 and 4.14]{HRV}.
\end{rem}

\begin{rem}
\label{rem:symbolicAss principalBorelsssi}
It follows from \eqref{eq:saturationAss} that, for an ideal $I$, 
the equalities $I^k = I^{(k)_{\Ass}}$ for $1 \leq k \leq \astab(I)$ imply that $I^k = I^{(k)_{\Ass}}$ for all $k \geq 1$. 

In particular, by \Cref{prop:persistence}(2) and \Cref{prop:astab polymatroidal}(2), a principal Borel symmetric strongly shifted ideal satisfies $I^k = I^{(k)_{\Ass}}$ for all $k \geq 1$ if and only if $I^k = I^{(k)_{\Ass}}$ for $1 \leq k \leq n-1$. While the latter statement is true for every polymatroidal ideal by \cite[Theorem 4.1]{HQureshi}, \Cref{thm:stableAss} shows that for a principal Borel sssi it suffices to check equality of powers and symbolic powers in a potentially smaller range, for $1 \leq k \leq \min \{ n-1, s(s-1)+1 \}$.    
    This addresses a question of Huneke, which is open for arbitrary monomial ideals. 
\end{rem}

In \cite{HVladoiu}, Herzog and Vladoiu define a monomial ideal to be of \emph{intersection type} if it can be decomposed as an intersection of powers of monomial primes. Principal Borel sssi's satisfy this property by \Cref{prop:intersection symbolic}. In fact, in \cite[Proposition 2.1]{HVladoiu} it is shown that every polymatroidal ideal is of intersection type. Exploiting this fact, in \cite[Corollary 3.5]{HVladoiu} the authors show that for every polymatroidal ideal $I$ generated in degree $d$, one has $I^{(dk)_{\Min}} \subseteq I^k $ for any $k \geq 1$. In particular, this applies to principal Borel sssi's. Our contribution below is to strengthen this containment.


\begin{prop}
\label{prop:containment}
 Let $I=\Sss(\l)$ be a principal Borel sssi, with $d=|\l|$ and $c=\min(\l)$. Then, $I^{(m)_{\Min}} \subseteq I^k $  whenever $m/k\geq d/\l_c$.   \end{prop}
\begin{proof}
Recall from \Cref{thm:irredundant}~(2) that $I^{(m)_{\Min}}=I_{n,c}^{(\l_cm)}$. Let $\mu\in \L(I_{n,c}^{(\l_cm)})$ and assume that  $m/k\geq d/\l_c$, i.e., $\l_cm\geq dk$.
Then, by \eqref{eq:symbolicgen} for each $j\geq c$ we have
\[
\sum_{i=1}^j\mu_i=\sum_{i=1}^c\mu_i+\sum_{i=c+1}^j \mu_i\geq \l_cm+(j-c)\frac{\l_cm}{c}=\frac{\l_cmj}{c}\geq \frac{dkj}{c}\geq dk\geq  \left(\sum_{i=1}^j \l_i\right) k,
\]
since $d=  \sum_{i=1}^n \l_i$.
The above inequality shows in view of  \eqref{eq:symbolicgen} that $I_{n,c}^{(\l_cm)}\subseteq I_{n,j}^{(a_j k)}$ for each $c\leq j\leq n$ and for $a_j=\sum_{i=1}^j \l_i$. It follows by \Cref{prop:intersection symbolic} that $I_{n,c}^{(\l_cm)}\subseteq I^k$.
\end{proof}

The above is particularly relevant to the {\em Containment Problem}, which asks for an ideal $I$ to determine the pairs $m,k$ so that $I^{(m)}\subseteq I^k$. This is an important and well studied problem in commutative algebra, which is open in its full generality. We refer the reader to \cite{{CEHH},{GrifoH},{MNBetancourt}} for known results in the case of monomial ideals.

\subsection{The intersection property of a principal Borel sssi}
\label{section:intersection}


The decomposition formula in \Cref{prop:intersection symbolic} mimics analogous decompositions for principal Borel ideals or transversal polymatroidal ideals; see \cite[Theorem 3.1]{FMS13} and \cite[Corollary 4.10]{HRV}. 
Inspired by these results, in \cite{BrunsConca}, Bruns and Conca defined an ideal $I$ to be $P$-adically closed if 
\[
I=\bigcap_{P\in \Ass(I)} P^{(v(P))}, \text{ where } v(P)=\max\{t : I\subseteq P^{(t)}\}.
\]
Since the associated primes of monomial ideals are generated by a subset of the variables, that is, a regular sequence, 
one obtains that a {\em monomial} ideal $I$ is {\bf $P$-adically closed} if 
\begin{equation}
\label{eq:Padic}
I=\bigcap_{P\in \Ass(I)} P^{v(P)}, \text{ where } v(P)=\max\{t : I\subseteq P^{t}\}.
\end{equation}
It follows from \Cref{prop:intersection symbolic} that principal Borel sssi's are $P$-adically closed, since for each $j$ the exponent $a_j$ coincides with $v(\sigma(P_j))$ for every $\sigma \in \sym_n$. 

We now seek to characterize $P$-adically closed symmetric monomial ideals.
%
%

\begin{prop}
\label{prop:$P$-adically closedisStronglyshifted}
A symmetric monomial ideal is $P$-adically closed if and only if it can be decomposed as 
\[
 I = \bigcap_{i=1}^t I_{n,c_i}^{(v_i)},
\] 
where $v_i=v(x_1,\ldots, x_{c_i})$. Moreover, all such ideals $I$ are  symmetric strongly shifted.

\end{prop}
\begin{proof}
It is clear from the definition that an ideal admitting such a decomposition is $P$-adically closed. Moreover, since each $I_{n,c_i}^{(v_i)}$ is symmetric strongly shifted by \cite[Theorem 4.3]{BDGMNORS}, then $I$ is symmetric strongly shifted by \Cref{prop:sum intersection}.

To prove the converse, we first show that a $P$-adically closed symmetric monomial ideal $I$ that is height-unmixed  with $\codim(I)=c$ is a symbolic power ideal of a square-free Veronese ideal. This is because $\Ass(I)$ is closed under the action of $\sym_n$ and the monomial primes of height equal to $\codim(I)$ form a single orbit under the action of $\sym_n$. Moreover the symmetry of $I$ yields that for $P=(x_1,\ldots,x_c)$ and each $\sigma\in \sym_n$ 
\[
v(P)=\max\{t : I\subseteq P^{t}\}=\max\{t : I\subseteq \sigma(P)^{t}\}=v(\sigma(P)).
\]
Thus equations \eqref{eq:Padic} and \eqref{eq:stconfig} yield for $I$ P-adically closed, symmetric and height-unmixed
\begin{equation}
\label{eq:unmixed}
I=\bigcap_{\sigma\in \sym_n} \sigma (x_1,\ldots, x_c)^v = \left(\bigcap_{\sigma\in \sym_n} \sigma (x_1,\ldots, x_c) \right)^{(v)}=I_{n,c}^{(v)}
\end{equation}
with $v=v (x_1,\ldots, x_c)$.

Next, suppose that $I$ is a $P$-adically closed symmetric monomial ideal that is not necessarily height-unmixed, with associated primes of distinct heights $c_1, \ldots, c_t$. Observe that for each $i$ the ideal $\displaystyle{\,\bigcap_{P\in \Ass(I) : \,\codim(P)=c_i} P^{v(P)}}$ is a $P$-adically closed symmetric ideal that is height-unmixed. Hence, equation \eqref{eq:unmixed} implies that setting $v_i=v(x_1,\ldots, x_{c_i})$ we obtain
\[
I = \bigcap_{i=1}^t \left( \bigcap_{P\in \Ass(I) : \,\codim(P)=c_i} P^{v(P)} \right) =  \bigcap_{i=1}^t I_{n,c_i}^{(v_i)}.
\]
\end{proof}

Comparing the previous result with \Cref{prop:intersection symbolic}, it would be natural to ask wether a $P$-adically closed sssi must be a principal Borel sssi. However, the following example shows that this is not necessarily true. This is in sharp contrast with the case of strongly stable ideals, which can only be written as intersections of powers of monomial prime ideals if they are principal Borel; see \cite[Proposition 2.8]{HVladoiu}.
\begin{ex}
The ideal $I=I_{5,1}\cap I_{5,3}^{(4)}$ is $P$-adically closed by \Cref{prop:$P$-adically closedisStronglyshifted} but it has partition Borel generating set $B(I)=\{(1,1,3,3,3),(1,2,2,2,2)\}$.
\end{ex}


We conclude this section by stating another interesting consequence of \Cref{prop:intersection symbolic}. Recall that an ideal $I$ is said to be \emph{sequentially Cohen-Macaulay} if there exists a filtration of $R$-modules
\[
D_0 = 0 \subsetneq D_1 \subsetneq \ldots \subsetneq D_s = R/I
\]
so that,  for all $1 \leq i \leq s$, $\dim(D_{i-1}) < \dim(D_{i})$ and the quotient modules $C_i = D_i / D_{i-1}$ are Cohen-Macaulay $R$-modules.
In particular, a Cohen-Macaulay ideal is sequentially Cohen-Macaulay, as it suffices to construct the $D_i$'s by going modulo a maximal regular sequence, one element at a time. 
The following result can be interpreted as a generalization of the fact that ideals of monomial star configurations are Cohen-Macaulay \cite[Proposition 2.9]{GHM}.
\begin{prop}
Let $I$ be a principal Borel sssi. Then, $I$ is sequentially Cohen-Macaulay.
\end{prop}
\begin{proof}
  By \Cref{prop:intersection symbolic}, $I$ can be decomposed as an intersection of symbolic powers of squarefreeVeronese ideals (also known as ideals defining monomial star configurations). The conclusion now follows from \cite[Proposition 3.1]{LinShen}.
\end{proof}

\section{Toric ideals and Rees algebras}
  \label{section:Rees}
  
  Given a finite set of monomials $G=\{m_1,\ldots, m_s\}\subset R$, the toric ring of $G$ is the subring $K[G] = K[m_1,\ldots,m_s]$ of $R$. Let $A=K[T_1, \ldots, T_s]$ denote a polynomial ring in $s$ new indeterminates over K and define a surjective homomorphism $\pi : A\twoheadrightarrow K[G]$ by  $\pi(T_i) = m_i$. The {\em toric ideal} of $G$  is $\I_G=\ker(\pi)$. In other words, the toric ideal of $G$ is the defining ideal of the toric ring $K[G]$. If $G=G(I)$ we often extend the terminology by referring to $K[G]$ and $\I_G$ as the toric ring and toric ideal of $I$. 
  
  Toric rings of equigenerated monomial ideals are coordinate rings of {\em projective toric varieties}. Indeed, the Zariski closure $X_G$ for the image of the map $\mathbb{P}^{n-1}\to \mathbb{P}^{s-1}$ given by $(x_1,\ldots, x_n)\mapsto(m_1,\ldots, m_s)$ has defining ideal $\I_G$ and coordinate ring $K[G]$.
  
Toric rings are better understood by considering the blow-up algebras of the ideal $I$.
  Recall that the \emph{Rees ring} of an ideal $I\subset R=K[x_1, \ldots, x_n]$, $\R(I)=\bigoplus_{i\geq 0}I^i t^i$,  is a quotient of a polynomial ring $S=R[T_1, \ldots, T_s]$ under a ring homomorphism $\varphi: S \twoheadrightarrow \mathcal{R}(I)$ given by $T_i\mapsto f_it$, with $f_i$ the $i$-the generator of $I$. 
  The map $\varphi$ factors through $\Sym(I)$, the symmetric algebra of $I$, as indicated in the diagram below. 

\begin{equation}
\label{eq:Rees}
\begin{tikzcd}[column sep=1.5em]
 & \Sym(I) \arrow{dr}{} \\
S \arrow{ur}{\varphi'} \arrow{rr}{\varphi} && \R(I)
\end{tikzcd}
\end{equation}
The \emph{fiber cone}, or \emph{special fiber ring}, of $I$ is the quotient $\cF(I)= \R(I)/(x_1,\ldots x_n) \cong \R(I) \otimes_R K$. If $G=G(I)$ then, in the notation above $A=S\otimes_R K$ and $\pi=\varphi\otimes_R K$, thus we recognize by comparing presentations that  $\cF(I)\cong K[G]$ is the toric ring of $G$.

%
%

From \eqref{eq:Rees} it follows that $\Sym(I)$, $\R(I)$ and $\cF(I)$ can be described as quotients of polynomial rings. Understanding the kernels of the maps $\varphi, \varphi'$ in \eqref{eq:Rees} allows to provide structure theorems for these algebras, identifying a presentation in terms of generators and relations. 

Since  $\K:=\mathrm{Ker}(\varphi) \supseteq \cL:= \mathrm{Ker}(\varphi')$, one always has that the relations defining the symmetric algebra of $I$ are also relations for the Rees algebra. In fact, it turns out that  $\cL$ consists of the elements of $\K$ that are linear in the variables $T_i$, which we write as $\cL = \K_{(*,1)}$ \cite{Vasconcelos}.
Moreover, if we denote $\J := \K\otimes_R K \subseteq K[T_1, \ldots, T_s]$, by construction it is clear that $\J S \subseteq \K$.  Thus, the relations of the fiber cone are also relations for the Rees algebra, whence $\K \supseteq \cL + \J S$.  An ideal $I$ is said to be of {\em fiber type} if equality holds. For $G=G(I)$, we have that $\J=\I_G$ is the toric ideal of $G$.  Every toric ideal is a prime ideal generated by binomials; see e.~g.~\cite[Proposition 10.1.1]{HHbook}. 


\subsection{Fiber type property}

In this section, we prove that every equigenerated symmetric strongly shifted ideal is of fiber type (see \Cref{thm:fibertype}). 
This property constitutes yet another similarity between symmetric strongly shifted ideals and strongly stable ideals, which are also of fiber type by \cite[Theorem 5.1 and Example 4.2]{HHV}.

A key ingredient in our proof is the fact that, for an equigenerated ideal $I$ one can define a grading on $S=R[T_1,\ldots, T_s]$ by setting 
\begin{equation}
\label{eq:grading}
\deg(r)=(\deg_R(r),0) \text{ for } r\in R \text{ and }  \deg(T_i)=(0,1).
\end{equation}
 With this grading, $\cF(I) \cong [\R(I)]_{(0,*)}$ and $\Sym(I)  \cong [\R(I)]_{(*,1)}$. Hence, $I$ is of fiber type if and only if the ideal $\K$ defining the Rees algebra is generated in bidegrees $(0,*)$ and $(*,1)$. 
We also crucially use the fact that equigenerated symmetric shifted ideals have {\em linear resolutions} by \cite[Theorem 3.2]{BDGMNORS}. The following lemma yields more information on their syzygies.

\begin{lem}
\label{lem:syzygy}
Let $I$ be an equigenerated symmetric shifted ideal. The syzygies on $I$ are generated by relations of the form
$
x_iu-x_{u_{\max}}v, 
$
 where $v\prec u \in G(I)$, $x_i\in C(u)$, and $\prec$ is defined in \Cref{defn:order} while $C(u)$ and $u_{\max}$ are defined in  \eqref{eq:C(u)}.
 \end{lem}
 \begin{proof}
 The proof utilizes the notation in  \eqref{eq:C(u)}. Set $J =(v\in G(I): v \prec u)$ and recall that $J:(u)=C(u)$ by \cite[proof of Theorem 3.2]{BDGMNORS}.
 A resolution for $I$ can be constructed as an (iterated) mapping cone from the resolution of $J$ and that of $C(u)$ utilizing the short exact sequence
 \[
 0\to R/(J:u)\to R/J \to R/(J+(u))\to 0.
 \] 
 In particular, this yields that the relations on $I$ are generated by the relations on $J$ together with the relations of the form
 $x_iu-w$ with $x_i\in C(u)$ and $w\in J$. Take $x_i\in C(u)$ and set $v=ux_i/x_{u_{\max}}$. From the symmetric shifted property of $I$ and the definitions of $C(u)$ and $u_{\max}$ in \eqref{eq:C(u)} one deduces that $v\in I$ and that $v\prec u$. Moreover, since $I$ is assumed equigenerated and $\deg(v)=\deg(u)$ it follows that $v\in G(I)$ and hence $v\in J$. 
 Since every relation $x_iu-w$ as above can be written as 
 \[
 x_iu-w=x_iu-x_{u_{\max}}v+(x_{u_{\max}}v-w),
 \]
 we deduce that the syzygies on $I$ are generated by the syzygies of $J$ and the set of relations $x_iu-x_{u_{\max}}v$.
 The desired conclusion follows by induction on the number of monomial generators of $I$.
 \end{proof}

\begin{thm}
\label{thm:fibertype}
 An equigenerated symmetric strongly shifted ideal is of fiber type. The defining relations of its Rees algebra are generated in bidegrees $(0,*)$ and $(1,1)$ with respect to the grading \eqref{eq:grading}. 
\end{thm}
\begin{proof}
Using the notation in \eqref{eq:Rees}, set $\K=\ker(\varphi)$ to be the set of relations of $\R(I)$ and set $\cL=\ker(\varphi')$ to be the set  of relations of $\Sym(I)$. Let $\J$ be the kernel of the map $\varphi\otimes_R K: K[T_1, \ldots, T_s] \twoheadrightarrow \cF(I)$. 
Our goal is to show that $\K=\cL+\J S$. The containment $\cL+\J S\subseteq \K$ being evident, 
we proceed to establish the opposite containment $\K\subseteq \cL+\J S$. 

Since $I$ is a monomial ideal, $\K$ is generated by homogeneous binomials. Consider a minimal generator $f$ for $\K$ of bidegree  $(d,k)$. If $d \neq 0$, $f$ corresponds to a minimal relation of degree $d$ on $I^k$. However, since $I$ is symmetric strongly shifted and equigenerated, the same is true of $I^k$ by \Cref{cor:powers} and thus by \cite[Theorem 3.2]{BDGMNORS} $I^k$ has a linear minimal free resolution. It then follows from the minimality of $f$ that $d=1$. 
Thanks to \Cref{lem:syzygy}, we may also assume that $f$ has the form
\[
f=x_{i}T_{i_1}\cdots T_{i_k}-x_{u_{\max}}T_{j_1}\cdots T_{j_k},
\] 
where $u=\varphi(T_{i_1}\cdots T_{i_k})=f_{i_1}\cdots f_{i_k}$ with $f_{i_\ell}\in G(I)$.  Set 
$v=\varphi(T_{j_1}\cdots T_{j_k})=f_{j_1}\cdots f_{j_k}$ and notice that $\varphi(f)=0$ implies $x_{i}u=x_{u_{\max}}v$. 

By definition of $C(u)$, since $x_i\in C(u)$, the exponent of the variable $x_{u_{\max}}$ in $u$ is larger than the exponent of $x_i$ in $u$. Thus, the same must be true for at least one $f_{i_\ell}$. Set $f_t=f_{i_\ell}x_i/x_{u_{\max}}$. Because $I$ is symmetric strongly shifted we have $f_t\in G(I)$ and $x_iT_{i_\ell}-x_{u_{\max}}T_t\in K$, which yields
\begin{eqnarray}
f &=&x_{i}T_{i_1}\cdots T_{i_k}-x_{u_{\max}}T_{j_1}\cdots T_{j_k} \nonumber \\
&=&\left(x_iT_{i_\ell}-x_{u_{\max}}T_t \right)T_{i_1}\cdots T_{i_{\ell-1}}T_{i_{\ell+1}} \nonumber\\
&&+x_{u_{\max}}\left(T_{i_1}\cdots T_{i_{\ell-1}}T_tT_{i_{\ell+1}}\cdots T_{i_k}-T_{j_1}\cdots T_{j_k}\right) \in \cL+\J S.
\label{eq:fibertype}
\end{eqnarray}
The above equation allows to conclude regarding the fiber type property. It also shows that $\K$ is generated by its elements of bidegrees $(1,1)$ and $(0,*)$.
\end{proof}

\begin{rem}
\label{rem:shiftedRees} 
In the proof of \Cref{thm:fibertype}, the assumption that $I$ is symmetric strongly shifted rather than just symmetric shifted was only used in order to apply \Cref{cor:powers}. This is because we do not currently know whether symmetric (not strongly) shifted ideals are closed under powers. In particular, a positive answer to \Cref{quest:powers shifted} would guarantee that equigenerated symmetric shifted ideals are of fiber type. 

On the contrary, the following example shows that a symmetric shifted ideal which is not equigenerated need not be of fiber type.
\end{rem}

\begin{ex}
Consider the symmetric strongly shifted ideal with $\Lambda(I)=\{(1,1,1), (0,2,2)\}$. (Note that $(1,1,2)\in P(I)$ since $(1,1,1)\in \Lambda(I)$.)  It is given by  
$I=\left(x_{1}x_{2}x_{3},\,x_{2}^{2}x_{3}^{2},\,x_{1}^{2}x_{3}^{2},\,x_{1}^{2}x_{2}^{2} \right)$ and it is not equigenerated.
 Calculations on Macaulay2 \cite{M2} show that the Rees algebra of $I$ has the following minimal presentation
 \[
 \R(I)=\frac{R[T_1, T_2, T_3, T_4]}{\left(x_{2}x_{3}T_{1}-x_{1}T_{2},\,x_{1}x_{3}T_{1}-x_{2}T_{3},\,x_{1}x_{2}T_{1}-x_{3}T_{4},\,x_{3}^{2}T_{1}^{2}-T_{2}T_{3},\,x_{2}^{2}T_{1}^{2}-T_{2}T_{4},\,x_{1}^{2}T_{1}^{2}- T_{3}T_{4}\right)}.
 \]
The last three listed relations of $\R(I)$ demonstrate that $I$ is not of fiber type. In particular, the equigeneration hypothesis is needed in \Cref{thm:fibertype}.
\end{ex}
  

  \subsection{The toric ideal of a principal Borel sssi}
  \label{section:Rees principalBorelsssi}
  In this subsection we focus on Rees algebras and fiber cones of principal Borel symmetric shifted ideals.  By \Cref{thm:polymatroidal} principal Borel sssi's $I$ are polymatroidal ideals and hence enjoy the symmetric exchange property as described in the comments following \Cref{def:polymatroidal ideal}. This property yields for each pair $r=x^u,s=x^v\in G(I)$ with $u_i>v_i$ an index $j$ and so that
  \[
  t:=x^ux_j/x_i\in G(I), w:=x^vx_i/x_j \in G(I), \text{ and thus } rs=tw.
  \]
  In terms of the toric ring of $I$, the last identity yields the identity $ T_rT_s=T_{t}T_{w}$ in $\cF(I)$. Equivalently, the binomial $ T_rT_s-T_{t}T_{w}$, termed a {\em symmetric exchange relation}, belongs to the defining ideal $\J$ of $\cF(I)$.
  
  Our first result shows that the defining ideal of the fiber cone and Rees algebra of a principal Borel sssi is generated by quadrics. In particular the toric ideal of a principal Borel sssi is generated by its exchange relations. 
  
Below we use the notation $\deg_i(m)$ to mean the exponent of $x_i$ in a monomial $m$.

\begin{thm}
\label{thm:quadraticFR}
  Let $I$ be a principal Borel sssi. Then the toric ideal of $G(I)$, also known as  the defining ideal of $\cF(I)$, is generated by quadrics, namely  the symmetric exchange relations
  \begin{equation}
  \label{eq:symexchange}
  T_rT_s-T_{t}T_{w}
\end{equation}
where  $r,s,t,w\in G(I)$ satisfy 
$\deg_i(r)> \deg_i(s), \deg_j(r)< \deg_j(s), t=rx_j/x_i \in G(I)$, and $w=sx_i/x_j \in G(I)$. 

Moreover, the defining ideal of $\R(I)$ is also generated by quadrics, specifically by the exchange relations  in \eqref{eq:symexchange}
together with the  relations
 \begin{equation}
  \label{eq:syzygyrelations}
x_iT_u-x_{u_{\max}}T_v,
\end{equation}
where $v\prec u \in G(I)$ cf.~\Cref{defn:order}, $x_i\in C(u)$, and $C(u)$ and $u_{\max}$ are as in  \eqref{eq:C(u)}.
\end{thm}
\begin{proof}
   It follows from \Cref{prop:Borel=product} that $I = \Sss(\{\l\})=\prod_{i=1}^n I_{n,i}^{\l_i-\l_{i-1}}$ is a product of square-free Veronese ideals. 
   Moreover, \Cref{prop:SEP} implies that each factor is a polymatroidal ideal which satisfies the strong exchange property.
   Now, notice that for any two polymatroidal ideals $J_1$ and $J_2$ with polymatroidal bases $B_1=G(J_1)$ and $B_2=G(J_2)$ the set 
   \[
   B_1 B_2 = \{ b_1 b_2 \, | \,  b_1 \in B_1, b_2 \in B_2\} 
   \]
  is a polymatroidal base for the polymatroidal ideal $J_1 J_2$, i.e., $G(J_1J_2)=B_1B_2$ (see \cite[Theorem 5.3]{ConcaH} and \cite[p.~4]{Nicklasson}). 
  Hence, $I$ is a polymatroidal ideal admitting a polymatroidal basis which is a product of bases with the strong exchange property. By \cite[Theorem 3.5]{Nicklasson}, one then has that the defining ideal of $\cF(I)$ is generated by the symmetric exchange relations. 
  
  The claim on the Rees algebra follows from the fiber type property of $I$, \Cref{thm:fibertype}, and in particular from the computations in equation \eqref{eq:fibertype} in the proof of this result. The fact that the defining ideal of $\R(I)$ is quadratic, without the detailed knowledge of the generators \eqref{eq:syzygyrelations} can also be deduced from  \cite[Theorem 5.2]{Nicklasson}.
   \end{proof}

 Since the principal Borel ideals are the symmetric polymatroidal ideals, we deduce the following result, which answers in the affirmative conjectures of White \cite{White} and Herzog--Hibi \cite{HHibi02} in the special case of symmetric polymatroids.
 \begin{conj}[\cite{White}, \cite{HHibi02}]
 \label{conj:White}
 For a polymatridal ideal $I$, the toric ideal of $G(I)$ is generated by the symmetric exchange relations.
 \end{conj}
 
 \begin{cor}
 Every symmetric polymatroidal ideal satisfies \Cref{conj:White}.
 \end{cor}
 \begin{proof}
 This follows from \Cref{thm:polymatroidal} and \Cref{thm:quadraticFR}.
 \end{proof}

Although the proof of \Cref{thm:quadraticFR} heavily utilizes the polymatroidal nature of principal Borel sssi's, we do not have any examples of equigenerated symmetric strongly shifted ideals whose fiber cone cannot be generated by quadrics. Therefore we ask:
\begin{quest}
\label{q1}
 Is the toric ideal of any equigenerated symmetric strongly shifted ideal quadratic?
\end{quest}

By contrast, the following example shows that, if $I$ is symmetric shifted but not strongly shifted, the defining ideal of the special fiber ring of $\cF(I)$ may not be generated by quadrics.
\begin{ex}
\label{ex:shiftednotquadratic}
   In $k[x_1, x_2, x_3, x_4]$, the equigenerated symmetric shifted ideal $I$ with 
   \[
   \Lambda (I) = \{ (1,1,2,2), (0,2,2,2), (0,1,2,3) \}
   \]
    is not strongly shifted; see \cite[Example 2.5]{BDGMNORS}. Moreover, Macaulay2 \cite{M2} shows that the defining ideal of  $\cF(I)$ contains 28 minimal cubic relations. 
\end{ex}

While \Cref{conj:White} is open for arbitrary polymatroidal ideals, it is indeed satisfied by several classes of polymatroidal ideals. These include, for instance, polymatroidal ideals satisfying the strong exchange property \cite[Therem 5.3(b)]{HHibi02}, principal Borel ideals \cite{{DeNegri},{HHibi02}}, lattice path polymatroidal ideals \cite[Theorem 2.10]{Schweig2}, and polymatroidal ideals satisfying the so called \emph{one-sided strong exchange property} \cite[Theorem 1.2]{Lu}. A version of  \Cref{conj:White} ``up to saturation" was settled in \cite{LasonMichalek}.

In all of the mentioned cases, the defining ideal of the special fiber ring $\cF(I)$ is in fact generated by a Gr\"{o}bner basis of quadrics. The latter condition is satisfied if the algebra generators of $I$ are {\em sortable} \cite{Sturmfels}, a condition which unfortunately does not necessarily hold for an arbitrary principal Borel sssi. When $\cF(I)$ is generated by a Gr\"{o}bner basis of quadrics, $\cF(I)$ is a Koszul algebra. Recall that a standard graded algebra $A$ over a field $K$ is \emph{Koszul} if the residue class field $A/K$ has a linear $A$-resolution. 

Another class of polymatroidal ideals whose fiber cone $\cF(I)$ is Koszul, is given by transversal polymatroidal ideals. For an ideal of this kind, in \cite[Theorem 3.5]{Conca} Conca proved that $\cF(I)$ is generated by quadratic polynomials, which however need not coincide with the symmetric exchange relations. In fact, to the best of our knowledge \Cref{conj:White} is open for arbitrary transversal polymatroidal ideals (we refer the reader to \cite{LasonMichalek} for a proof for transversal matroidal ideals).

The following result provides classes of principal Borel sssi with Koszul toric ring. 
\begin{cor}
\label{cor:Koszul}
   Let $I = \Sss(\{ \l \})$ be a principal Borel symmetric strongly shifted ideal. Suppose that $\l$ is of one of the following types:
   \begin{enumerate}
    \item $\l= (a, \ldots, a)$ for some $a \neq 0 \in \N$;
    \item $\l= (a^s, b^{n-s})$ for some $a < b \in \N$, $s>0$;
    \item $\l= (a^s, b, c^{n-s-1})$ for some $a < b < c \in \N$, $s>0$.
    \item $\l$ satisfies $\Delta^i(\l)_1\geq 0$ for all $1\leq i\leq n$.
  \end{enumerate}
 Then, the toric ring of $I$, equivalently, the special fiber ring $\cF(I)$ is a Koszul algebra.
 \end{cor}
 \begin{proof}
    By \Cref{prop:SEP}, if $\l$ is of one of the first three given types, then $I$ satisfies the strong exchange property. The conclusion now follows from \cite[Therem 5.3(b)]{HHibi02}. If $\l$ has the fourth listed property, \Cref{prop:transversal} yields that $\Sss(\{\l\})$ is transversal, so the desired conclusion follows from \cite[Theorem 3.5]{Conca}.
     \end{proof}
     
     
It is known that high Veronese subrings of graded rings are Koszul \cite{Backelin, EisenbudReevesTotaro}. In this vein, we can establish the Koszul property of principal sssi's up to taking a sufficiently high multiple of the partition Borel generator.
 \begin{prop}
\label{prop:Koszuluptomultiple}
Let $\l\in P_n$ be a partition. Then for sufficiently large  integers $k$ the toric ring of the ideal $I= \Sss(\{ k\l \})$ has a quadratic Gr\"obner basis. In particular, the toric ring of $I$, $\cF(I)$, is a Koszul algebra.
\end{prop}
\begin{proof}
Recall that $I= \Sss(\{ k\l \})= \Sss(\{\l \})^k$ and hence $\cF(I)=\bigoplus_{i\geq 0}  \Sss(\{\l \})^{ki}$ is the $k$-th Veronese subring of $ T=\cF(\Sss(\{\l \}))$. It is established in \cite[Theorem 2]{EisenbudReevesTotaro} that the defining ideal $\J$ of $\cF(I)$ has an initial ideal generated in degree $\leq \max\{\lceil \reg(T))/k\rceil, 2\}$, where $ \reg$ denotes the Castelnuovo-Mumford regularity. Therefore the initial ideal of $\J$ has a quadratic Gr\"obner basis whenever $ k\geq \reg(T)/2$.
\end{proof}

While not all principal Borel sssi's satisfy the assumptions of \Cref{cor:Koszul}, we do not currently know of any examples of principal Borel sssi's whose fiber cone are not Koszul. Thus we pose the following question.
\begin{quest}
\label{q2}
 Is the toric ring of any principal Borel symmetric strongly shifted ideal Koszul?
\end{quest}

The answer to this question is object of ongoing work of Kuei-Nuan Lin and Yi-Huang Shen, as we learned via personal communication while this manuscript was being written.

\medskip

We conclude this section by describing the geometry of the toric rings associated to principal Borel sssi's. A convex lattice polytope $\bP$ is said to be {\em normal} or to have the {\em integer decomposition property} if it satisfies the following condition: given any positive integer $d$, every lattice point of the dilation $d\cdot \bP$, can be written as the sum of exactly $d$ lattice points in $\bP$. Let $I$ be the ideal generated  by all monomials with exponents in $\bP$. Normality of $\bP$ is equivalent to $\overline{I^d}=I^d$ for positive integers $d$, hence to $I$ being normal. We thus obtain the following consequence of \Cref{prop:normal}.

\begin{cor}
\label{cor:normalpolytope}
For each $\l\in P_n$ the permutohedron $\bP(\l)$ is a normal polytope.
\end{cor}
\begin{proof}
This follows from \Cref{prop:permutohedron} and \Cref{prop:normal}.
\end{proof}

 A normal lattice polyhedron $\bP$ uniquely determines a projective toric  variety $X_{\bP}$ by means of its normal fan; see \cite[Definition 2.3.14]{CLO}. We term the toric variety defined by the permutohedron $\bP(\l)$ with respect to the lattice $\Z^n/{\rm span}(1,\ldots, 1)$ the {\em permutohedral toric variety} $X_{\bP(\l)}$. The homogeneous coordinate ring for the image of the projective embedding $X_{\bP(\l)}\hookrightarrow \mathbb{P}^{N}$ given by the  divisor $D_{\bP(\l)}$ is in our notation $K[G(\Sss(\l))]=\cF(\Sss(\l))$ and hence the defining equations of  $X_{\bP(\l)}$ in this embedding are given by the toric ideal of $\Sss(\l)$. Since $\bP(\l)$ is normal, $D_{\bP(\l)}$ is very ample and $X_{\bP}$ is projectively normal. 
 Questions \ref{q1} and \ref{q2} arise naturally for this class of algebraic sets. Our work yields the following answer.
 
 \begin{cor}
 \label{cor:quadraticpermutohedral}
 For any $\l\in P_n$, the defining ideal of the permutohedral toric variety $X_{\bP(\l)}$ is generated by quadratic polynomials. If $\l$ is of one of the types described in \Cref{cor:Koszul}, then $X_{\bP(\l)}$ has a Koszul coordinate ring. 
   \end{cor}
 \begin{proof}
 The claims follows from \Cref{thm:quadraticFR} and \Cref{cor:Koszul}, since the coordinate ring of $X_{\bP(\l)}$ is $\cF(\Sss\{\l\}$.
 \end{proof}
 
 The case $\l=(0,1,\ldots, n-1)$, which yields the standard permutohedron  and the (standard) {\em permutohedral variety} $X_{A_n}$, has been studied extensively from the point of view of its intersection theory \cite{Huh, HuhKatz} in connection with matroid theory. While generalized permutohedral varieties have been considered for various root systems, toric permutohedral varieties in the generality defined above as well as their coordinate rings seem to be currently unexplored.
 
One can extract several numerical invariants for permutohedral toric varieties and hence for toric rings of principal Borel sssi's from related invariants of the permutohedra.

\begin{rem} Consider a partition $\l\in P_n$.
\begin{enumerate}
\item The Hilbert function of $K[G(\Sss(\l))]=\cF(\Sss(\l))$ is the Erhart function of $\P(\l)$, namely $d \mapsto H(d):=$ the number of integer points in $d\cdot P(\l)$. If $\l=(0,1,\ldots, n-1)$, then $H(d)$ is the number of forests on $n$ vertices with $i$ edges \cite[Example 3.1]{Stanley}.
\item The degree of $X_{\bP(\l)}$ and the Hilbert-Samuel multiplicity of $K[G(\Sss(\l))]=\cF(\Sss(\l))$ are given by the normalized volume $\frac{{\rm Vol}(\bP(\l))}{(n-1)!}$. Formulae for the volume of a permutohedron can be found in \cite{Postnikov}. For instance, if $\l=(0^{n-d},1^d)$ then $\frac{{\rm Vol}(\bP(\l))}{(n-1)!}$ is the {\em Eulerian number}, that is the number of permutations of size $n-1$ with $d-1$ descents. For an arbitrary principal Borel sssi, the volume of $\bP(\l)$ is then calculated in terms of the {\em mixed Eulerian numbers}, i.e., normalized mixed volumes of the hypersimplices \cite[Proposition 9.8 and Definition 16.1]{Postnikov}.
\end{enumerate}
\end{rem}

\bigskip

\bibliography{biblio}
\bibliographystyle{alpha}

\end{document}